\documentclass[11pt,final]{article}
\usepackage{amsmath,amsthm,amsfonts,amssymb,graphicx,hyperref,enumerate,psfrag}
\usepackage{fullpage}

\usepackage[utf8]{inputenc} 

\newtheorem{lemma}{Lemma}

\newtheorem{remark}{Remark}
\newtheorem{proposition}{Proposition}
\newtheorem{theorem}{Theorem}

\newtheorem{corollary}{Corollary}

\newcommand{\EE}{{\mathbb{E}}}
\newcommand{\mm}{{\mathfrak{m}}}

\newcommand{\PP}{{\mathbb{P}}}

\newcommand{\dZ}{\mathbb {Z}}

\newcommand{\dR}{\mathbb {R}}

\newcommand{\cN}{\mathcal {N}}
\newcommand{\cO}{\mathcal {O}}
\newcommand{\cS}{\mathcal {S}}

\newcommand{\cG}{\mathcal {G}}

\newcommand{\cL}{\mathcal {L}}
\newcommand{\cP}{\mathcal {P}}

\newcommand{\bX}{X}
\newcommand{\bY}{Y}
\newcommand{\bZ}{Z}

\newcommand{\qq}{{\mathfrak q}}

\newcommand{\tmix}{t_{\textsc{mix}}}

\newcommand{\gap}{{\textbf{gap}}}

\newcommand{\bu}{{\mathfrak{v}}}
\newcommand{\bv}{{\bf{v}}}
\newcommand{\bw}{{\mathfrak{w}}}
\title{Cutoff for the mean-field zero-range process}
\author{Mathieu Merle, Justin Salez}

\begin{document}
\maketitle
\begin{abstract}
We study the mixing time of the unit-rate zero-range process  on the complete graph, in the regime where the number $n$ of sites tends to infinity while the density of particles per site stabilizes to some limit $\rho>0$. We prove that the worst-case total-variation distance to equilibrium drops abruptly from $1$ to $0$ at time $n\left(\rho+\frac{1}{2}\rho^2\right)$. More generally, we determine the mixing time from an arbitrary initial configuration. The answer turns out to depend on the largest initial heights in a remarkably explicit way. The intuitive picture is that the system separates into a slowly evolving solid phase and a quickly relaxing liquid phase. As time passes, the solid phase  {dissolves} into the liquid phase, and the mixing time is essentially the time at which the system becomes completely liquid. Our proof combines meta-stability, separation of timescale, fluid limits, propagation of chaos, entropy, and a spectral estimate by Morris (2006).
\end{abstract}
\tableofcontents
\section{Introduction}
\subsection{Model and results}
Introduced by Spitzer in 1970 \cite{MR0268959}, the \emph{zero-range process} has now become a classical model of interacting random walks. In its most general form, every site of a graph $G$ is allowed to contain an arbitrary number of indistinguishable particles, which randomly hop along the edges at a rate that only depends on the number of particles occupying the site of departure. The present paper is concerned with the  \emph{mean-field} setting where $G$ is simply the complete graph of order $n$ and at unit rate, each non-empty site expels a particle to a uniformly chosen site. Formally, the state space is 
\begin{eqnarray*}
\Omega & := & \left\{\eta \in \dZ_+^n\colon \eta_1+\cdots+\eta_n=m\right\},
\end{eqnarray*}
where $m$ represents the total number of particles in the system, and $\eta_i$ the number of particles occupying site $i$. The Markov generator $\cL$ acts on observables $f\colon \Omega\to\dR$ as follows:
\begin{eqnarray}
\label{def:generator}
(\cL f)(\eta) & := & \frac{1}{n}\sum_{1\le i,j\le n}{\bf 1}_{(\eta_i>0)}\left(f(\eta+\delta^j-\delta^i)-f(\eta)\right),
\end{eqnarray}
where $\delta^i_k$ equals $1$ if $k=i$ and $0$ else. This generator is irreducible and symmetric. Consequently, the uniform law $\pi$ on $\Omega$ is reversible, and  the process \emph{mixes}: the transition kernel $P_t=e^{t\cL}$ satisfies
\begin{eqnarray*}
P_t(\eta,\cdot) & \xrightarrow[t\to\infty]{} & \pi,
\end{eqnarray*}
regardless of the choice of the initial state $\eta\in \Omega$.  A standard way to quantify the rate at which this convergence to equilibrium occurs consists in estimating the so-called \emph{mixing time}:
\begin{eqnarray*}
\tmix(\eta;\varepsilon) & := & \min\left\{t\ge 0\colon \|P_t(\eta,\cdot)-\pi\|_{\textsc{tv}}\le \varepsilon\right\}.
\end{eqnarray*}
In this formula, $\varepsilon\in(0,1)$ is a parameter controlling the desired precision, and $\|\mu-\nu\|_{\textsc{tv}}:=\max_{A\subseteq \Omega}\left|\mu(A)-\nu(A)\right|$ denotes the total-variation distance between $\mu$ and $\nu$. Of particular interest is the \emph{worst-case  mixing time}, obtained by maximizing over all possible initial states $\eta\in\Omega$: 
\begin{eqnarray*}
\tmix(\varepsilon) & := & \max\left\{\tmix (\eta;\varepsilon)\colon \eta\in\Omega\right\}.
\end{eqnarray*}
Estimating this fundamental parameter -- and in particular, its precise dependency in $\varepsilon$ -- is in general a challenging task, see the books \cite{MR2341319,MR3726904}. The present paper is concerned with the regime where $n$ tends to infinity, while the density of particles per site stabilizes to some value $\rho\in(0,\infty)$:
\begin{eqnarray}
\label{assume:sparse}
n\to\infty, & \qquad & \frac{m}{n}\to \rho.
\end{eqnarray}
All asymptotic statements will be understood in this sense, and we shall  often keep the dependency upon $n$ implicit in order to lighten the notation. In the regime (\ref{assume:sparse}), a spectral gap estimate due to Morris \cite{MR2271475} implies that $\tmix(\eta;\varepsilon)=\cO(n)$. Here we determine the precise prefactor and express it explicitly in terms of the largest values of $\eta$. To describe the precise result, let us first note that, by symmetry, the initial heights may always be assumed to be arranged in decreasing order:
\begin{eqnarray}
\label{assume:order}
\eta_{1}\ \ge\ \eta_{2}\ \ge\ \ldots\ \ge\ \eta_{n}.
\end{eqnarray}
Passing to a subsequence, we may further assume without loss of generality that for each $k\ge 1$, 
\begin{eqnarray}
\label{assume:profile}
\frac{\eta_{k}}{n} & \xrightarrow[n\to\infty]{}  & u_k,
\end{eqnarray}
for  some non-increasing sequence $(u_k)_{k\ge 1}$ of non-negative numbers. With this standardized setting in mind, our main result can be stated in the following simple way. 
\begin{theorem}[Mixing times]\label{th:main}In the regime (\ref{assume:sparse})-(\ref{assume:order})-(\ref{assume:profile}), we have for each fixed $\varepsilon\in(0,1)$, 
\begin{eqnarray}
\label{eq:main}
\frac{\tmix(\eta;\varepsilon)}{n} & \xrightarrow[n\to\infty]{} & (1+\rho)u_1-\frac 12\sum_{i=1}^\infty u_i^2.
\end{eqnarray}
\end{theorem}
Note that by Fatou's Lemma, the limiting heights $(u_i)_{i\ge 1}$ must necessarily satisfy
\begin{eqnarray}
\label{ineq:constraint}
\sum_{i=1}^\infty u_i & \le & \rho.
\end{eqnarray}
Under this constraint, the right-hand side of (\ref{eq:main}) is uniquely maximized by taking $u_1=\rho$ and $u_2=u_3=\ldots=0$. Thus, the worst-case mixing time is achieved (at least to first order) by initially placing all particles on the same site, a fact which seems rather intuitive but for which we were not able to find a direct argument. As a consequence, we obtain the following important corollary. 
\begin{corollary}[Worst-case mixing time and cutoff]\label{co:worst}For any fixed $\varepsilon\in(0,1)$, in the regime (\ref{assume:sparse}), 
\begin{eqnarray}
\frac{\tmix(\varepsilon)}{n} & \xrightarrow[n\to\infty]{} & \rho+\frac{1}{2}\rho^2.
\end{eqnarray}
\end{corollary}
The remarkable fact that the precision parameter $\varepsilon\in(0,1)$ is absent from the limit adds the mean-field zero-range process to the growing list of  chains exhibiting what is known as a \emph{cutoff} \cite{MR1374011}:  instead of decaying gradually, the  total-variation distance to equilibrium stays close to $1$ until the mixing time, and then abruptly drops to $0$  over a much shorter timescale. We suspect the cutoff width to be here $\Theta(\sqrt{n})$, with a Gaussian profile in the limit. However,  our estimates are not precise enough  to establish this second-order refinement, which we leave as a conjecture. 
\subsection{The solid-liquid heuristic}
\emph{Condensation} is one of the most remarkable features of the zero-range process. In our setting, the total rate at which particles are expelled from a given site is $1$, regardless of the number $k$ of particles occupying that site. Consequently, the effective rate at which each particle is expelled is $1/k$ only: denser regions evolve more slowly. This simple observation naturally leads to a formal decomposition of the system into two components, or \emph{phases}, relaxing on very different timescales:
\begin{itemize}
\item A (slow) \emph{solid} phase, consisting of those few sites which are occupied by $\Theta(n)$ particles.
\item A (quick) \emph{liquid} phase, formed by those sites that are occupied by $o(n)$ particles.    
\end{itemize}
The presence of a solid phase is a clear indication that the system is out of equilibrium, since under the uniform distribution, the maximum occupancy is only logarithmic in $n$. The case $u_1=0$ in Theorem \ref{th:main} indicates that the converse is also true: in the absence of a solid phase, the system reaches equilibrium in negligible time. The proof of this fact occupies a substantial part of the paper. In light of it, the picture becomes much clearer: as time passes, the solid phase described by the profile (\ref{assume:profile}) progressively \emph{dissolves} into the liquid phase, and the mixing time is essentially the time at which the system becomes completely liquid. Note that the dissolution occurs on a time-scale of order $n$, since the effective jump rate per particle in the solid phase is $\Theta\left(\frac 1n\right)$. 

To obtain the precise prefactor appearing in the right-hand side of (\ref{eq:main}), we need to estimate the instantaneous melting rate of a solid site. In our mean-field setting, this is precisely the proportion of empty sites in the system, which in turns depends on the density of the liquid phase. What makes the problem tractable, despite this cyclic interaction between the two phases, is a  \emph{separation of timescales} phenomenon: the liquid phase relaxes so quickly that, on the relevant timescale, the solid phase may be considered as inert. Consequently, the liquid phase is permanently maintained in a \emph{metastable} state which resembles the true equilibrium, except that its  density is lower because a macroscopic number of particles are  ``stuck'' in the solid phase. This imposes a simple asymptotic relation between  the number of particles in the solid phase and the proportion of empty sites. As a consequence, the evolution of the solid phase can be approximated by an autonomous system of differential equations, whose explicit resolution yields the precise formula appearing in Theorem \ref{th:main}.

\subsection{Proof outline}
To make the above picture rigorous, we proceed in three steps, each occupying a  whole section. 
In Section \ref{sec:liquid}, we get a rough idea of the system by ignoring the precise {geometry} of the zero-range process $(\eta(t)\colon t\ge 0)$ and focusing on the distribution of the number of particles on a \emph{typical} site. This data is encoded into the so-called \emph{empirical distribution} of the system:
\begin{eqnarray}
\label{def:empirical}
Q(t) & := & \frac{1}{n}\sum_{i=1}^n\delta_{\eta_i(t)}.
\end{eqnarray}
For convergence purposes, we regard $\cP(\dZ_+)$ as a subset of $\ell^1(\dZ_+)$, with norm
$
\|q\| := \sum_{k=0}^\infty|q_k|.
$
At equilibrium, the empirical distribution is simple: if $\xi$ is uniform on $\Omega$, then in the regime (\ref{assume:sparse}),
\begin{eqnarray}
\label{equilibrium}
\left\|\frac{1}{n}\sum_{i=1}^n\delta_{\xi_i}-\cG(\rho)\right\| & \xrightarrow[n\to\infty]{\PP} & 0,
\end{eqnarray}
where $\xrightarrow[]{\PP}$ denotes convergence in probability, and  $\cG(\rho)$ the geometric distribution with mean $\rho$, i.e.
\begin{eqnarray*}
\cG_k(\rho) & = & \frac{1}{1+\rho}\left(\frac{\rho}{1+\rho}\right)^k.
\end{eqnarray*}
To discuss the $n\to\infty$ limit of the process $(Q(t)\colon t\ge 0)$, it will be convenient to  assume that 
\begin{eqnarray}
\label{assume:empirical}
Q(0) & \xrightarrow[n\to\infty]{} & q,
\end{eqnarray}
for some $q\in\cP(\dZ_+)$.  It turns out that this suffices to guarantee the convergence of the whole process $(Q(t)\colon t\ge 0)$. Moreover, the limit  $(\qq(t)\colon t\ge 0)$ is deterministic and characterized by the initial data $\qq(0)=q$ through the following explicit (non-linear) dynamics:
\begin{eqnarray}
\label{def:fluid}
\frac{d\qq_k}{dt} & = & \qq_{k+1}-\qq_k{\bf 1}_{(k\ge 1)}-\left(\sum_{\ell\ge 1}\qq_\ell\right)\left(\qq_k-\qq_{k-1}{\bf 1}_{(k\ge 1)}\right).
\end{eqnarray}
In the fluid limit literature, results of this type are referred to as  \emph{propagation of chaos} \cite{MR1108185}.
\begin{proposition}[Propagation of chaos]\label{pr:chaos}Under assumptions (\ref{assume:sparse}) and (\ref{assume:empirical}), we have
\begin{eqnarray*}
\sup_{t\in[0,T]}\left\|Q(t)-\qq(t)\right\|  & \xrightarrow[n\to\infty]{\PP} & 0,
\end{eqnarray*}
for any fixed horizon $T\ge 0$,  where  $(\qq(t)\colon t\ge 0)$ is the unique solution to (\ref{def:fluid}) with $\qq(0)=q$.
\end{proposition}
Metastability will then consist in showing that the fluid limit $\qq(t)$ relaxes as $t\to\infty$ towards a geometric profile as in (\ref{equilibrium}), except that $\rho$ is replaced by the \emph{tilted} density
\begin{eqnarray}
\label{def:lambda}
\lambda & := & \sum_{k=1}^\infty kq_k.
\end{eqnarray}
\begin{proposition}[Relaxation for the fluid limit]\label{pr:fluid}
We have $\qq(t) \xrightarrow[t\to\infty]{} \cG(\lambda)$.
\end{proposition}
Entropy will play a crucial role in the proof of this result. Note  that by Fatou's Lemma, we always have $\lambda\le\rho$, with strict inequality in the presence of a solid phase. We emphasize that time has not been rescaled with $n$ here:  the empirical distribution $Q(t)$ approaches the metastable equilibrium $\cG(\lambda)$ on a timescale $\Theta(1)$ only.
In Section \ref{sec:fast}, we build upon the above results to establish the case $u_1=0$ of Theorem \ref{th:main}, which ensures fast mixing  in the absence of a solid phase, i.e. when
\begin{eqnarray}
\label{assume:liquid}
\max_{1\le i\le n}\eta_i & = & o(n).
\end{eqnarray}
\begin{proposition}[Fast mixing]In the regime (\ref{assume:sparse})-(\ref{assume:liquid}), we have $\tmix(\eta;\varepsilon)=o(n)$.
\label{pr:fast}
\end{proposition}
In a sense, Propositions \ref{pr:chaos} and \ref{pr:fluid} already indicate this: if $Q(0)$ is uniformly integrable, then $\lambda=\rho$ and therefore, $Q(t)$ can be made arbitrarily close to the equilibrium profile $\cG(\rho)$ by choosing $t$ large, independently of $n$. This is, however, much weaker than  Proposition \ref{pr:fast} in three respects:
\begin{enumerate}[(i)]
\item The assumption (\ref{assume:liquid}) is far from ensuring that $\lambda=\rho$: the choice $\eta_1=\ldots=\eta_k=\frac{m}{k}$ and $\eta_{k+1}=\ldots=\eta_n=0$ with $1\ll k \ll n$ does satisfies $\max \eta=o(n)$, and yet $\lambda=0$ !  
\item  The empirical distribution $Q$ says nothing about the positions of the particles: if the system is exactly at equilibrium and we re-arrange the particles so that $\eta_1\ge\ldots\ge \eta_n$, then $Q$ is unchanged and yet the law of the system becomes asymptotically singular to $\pi$ ! 
\item 
The  convergence $\|Q(t)-\cG(\rho)\|\to 0$ is still far too weak to imply that the law of $Q(t)$ is close to equilibrium in total variation: moving $o(n)$ particles in an arbitrary way will not affect the convergence $\|Q(t)-\cG(\rho)\|\to 0$ and yet, changing the maximum occupancy from $\Theta(\log n)$ to anything larger already suffices to make the law of $Q(t)$ singular to equilibrium. 
\end{enumerate}
Finally, in Section \ref{sec:solid}, we provide the following description for the dissolution of the solid phase.
\begin{proposition}
\label{pr:densefluid}
In the regime (\ref{assume:sparse})-(\ref{assume:order})-(\ref{assume:profile}), we have for any fixed $T\ge 0$ and $i\ge 1$, 
\begin{eqnarray*}
\sup_{t\in[0,T]}\left|\frac{\eta_i(nt)}{n}-\bu_i(t)\right| & \xrightarrow[n\to\infty]{\PP} &  0,
\end{eqnarray*}
where the functions $\bu_1(t),\bu_2(t),\ldots$ are deterministic and satisfy
\begin{eqnarray}
\label{fluid:dense}
\bu_i(t) & = & \left(u_i -\int_0^t\frac{1}{1+\rho-\sum_{j=1}^\infty \bu_j(s)}ds\right)_+.
\end{eqnarray}  
\end{proposition}
Note that the Cauchy problem (\ref{fluid:dense}) is slightly degenerate, since the usual Lipschitz condition does not apply. We start by verifying that there is a unique solution  $\bu$ to this problem, and then show that the latter does indeed describe the evolution of the solid phase. In addition, we compute the time at which this solution vanishes, and find that it is precisely the right-hand side of (\ref{eq:main}). When combined with Proposition \ref{pr:fast}, this observation easily leads to the proof of Theorem \ref{th:main}.
\subsection{Related works}

The zero-range process has a long history. In the classical setting, the particles evolve on an infinite transitive graph like the lattice $\dZ^d$, and the description of the set of stationary laws constitutes by itself an important question. More recently, hydrodynamic limits and complex phenomena such as metastability and  condensation have received a considerable attention in both the mathematical and physical communities. The works are too numerous to be all cited, and  we refer the interested reader to the comprehensive survey \cite{MR2145800} and the references therein. 

Results addressing the rate of convergence to equilibrium of the zero-range process on finite graphs  are more limited. In \cite{MR2322692}, Caputo and Posta estimate the entropy dissipation constant on the complete graph in the condensation-free regime where the jump rate grows roughly linearly with the number of particles on the site. More directly related to our setting is an important work of Morris \cite{MR2271475}, in which the spectral gap of the constant-rate zero-range process is estimated on the complete graph and the $d-$dimensional torus. While the spectral gap provides general bounds on the mixing times, these are usually too crude to get the precise prefactor and establish cutoff. Nevertheless, the result of Morris plays an important role in our proof of fast mixing in the absence of a solid phase, see Section \ref{sec:fast}. Another important inspiration for the present work is a paper of Graham \cite{MR2521877} concerning the asymptotic behavior of  $(Q(t)\colon t\ge 0)$ in the special case where the initial configuration $\eta$ is constant. Our propositions \ref{pr:chaos} and \ref{pr:fluid} extend these results to arbitrary initial conditions. As explained above, this fluid limit only provides a very rough description of the system and much more work is needed in order to control the total variation distance to equilibrium. 

The first occurrences of a cutoff phenomenon were discovered in the 80's by Aldous, Diaconis and Shahshahani \cite{MR626813,MR770418,MR841111} in the case of card shuffling. Since then, other instances have been found in a variety of contexts and, notably, interacting particle systems. Three emblematic examples are the stochastic Ising model (on the complete graph  \cite{MR2550363}, the lattice \cite{MR3020173} and other topologies \cite{MR3193965}), the East process \cite{MR3320314}, and the exclusion process (on the complete graph \cite{MR2869447}, the line \cite{MR3474475}, and the cycle \cite{MR3551201,MR3689972}). Interestingly, the proof of cutoff for the exclusion process on the cycle implies that of the zero-range process on the cycle, via a well-known bijection \cite{EVANS2000}. To the best of our knowledge, the cycle is the only graph on which the zero-range process has been shown to exhibit cutoff. Extending this to the $d-$dimensional torus for $d\ge 2$ seems to constitute a natural and challenging problem. More generally, the question of characterizing the Markov chains that exhibit cutoff has attracted much attention over the past three decades, but remains unsolved.

\section{Metastability of the liquid phase}
\label{sec:liquid}

Before we establish Propositions \ref{pr:chaos} and \ref{pr:fluid}, let us briefly prove the statement (\ref{equilibrium}) for completeness. Let $\cN(n,m)$ denote the number of ways to place $m$ indistinguishable particles into $n$ sites:
\begin{eqnarray}
\label{binomial}
\cN(n,m) & := & { m+n-1 \choose n-1 }.
\end{eqnarray}
If $\xi$ is uniformly distributed on $\Omega$, we have for each $k\in\dZ_+$, 
\begin{eqnarray*}
\PP\left(\xi_1=k\right) \ = \ \frac{\cN(n-1,m-k)}{\cN(n,m)} & \textrm{ and } & 
\PP\left(\xi_1=k,\xi_2=k\right) \ = \ \frac{\cN(n-2,m-2k)}{\cN(n,m)}.
\end{eqnarray*} 
In the regime (\ref{assume:sparse}), these ratios tend to $\cG_k(\rho)$ and $(\cG_k(\rho))^2$, respectively. Thanks to the exchangeability of $(\xi_1,\ldots,\xi_n)$, this easily implies that 
\begin{eqnarray*}
Q_k & \xrightarrow[n\to\infty]{L^2} & \cG_k(\rho).
\end{eqnarray*}
Since $k\in\dZ_+$ is arbitrary, the claim follows.
\subsection{Propagation of chaos}
In this section, we establish Proposition \ref{pr:chaos}. We rely on the standard theory of hydrodynamic limits for Markov processes, a widely studied topic which is discussed at length in the comprehensive book of Ethier and Kurtz \cite{MR838085}. Define a map $F\colon\ell^1(\dZ_+)\to \ell^1(\dZ_+)$ by the formula
\begin{eqnarray*}
F_k(q) & = & q_{k+1}-q_k{\bf 1}_{(k\ge 1)}-\left(\sum_{\ell\ge 1}q_\ell\right)\left(q_k-q_{k-1}{\bf 1}_{(k\ge 1)}\right).
\end{eqnarray*}
This map is locally Lipschitz continuous: for any $q,q'\in \ell^1(\dZ_+)$, 
\begin{eqnarray}
\label{lipschitz}
\|F(q)-F(q')\| & \le & 2\left(1+\|q\|+\|q'\|\right)\|q-q'\|.
\end{eqnarray}
Consequently, for each $q\in \ell^1(\dZ_+)$, the Picard–Lindelöf Theorem ensures existence and uniqueness of a maximal $\ell^1(\dZ_+)-$valued solution $\left(\qq(t)\colon t\in[0,T_*)\right)$ to the  Cauchy problem
\begin{eqnarray}
\label{cauchy}
\qq(t) & = & q+\int_0^tF(\qq(s))ds.
\end{eqnarray}
Note, however, that the horizon $T_\star$ needs not a priori be infinite, as we have not yet ruled out the possibility that $\|\qq(t)\|$ explodes in finite time. Let us now show that the empirical distribution of the system satisfies an approximate version of (\ref{cauchy}). Thanks to our mean-field setting, the projected process $(Q(t)\colon t\ge 0)$ is again a Markov process on (a finite part of) $\cP(\dZ_+)$, with jumps
\begin{eqnarray}
\label{transitions:Q}
q & \mapsto & q+\frac{1}{n}\left(\delta^{\ell+1}+\delta^{k-1}-\delta^\ell-\delta^k\right)
\end{eqnarray}
occurring at rate ${\bf 1}_{(k\ge 1)}q_k\left(nq_\ell-{\bf 1}_{(\ell=k)}\right)$, for each $(k,\ell)\in\dZ_+^2$. The infinitesimal drift $D\colon \ell^1(\dZ_+)\to\ell^1(\dZ_+)$ can thus be decomposed as 
\begin{eqnarray}
\label{drift}
D  & = & F +\frac{1}{n}R ,
\end{eqnarray}
with $
R_k(q) =  2q_k{\bf 1}_{(k\ge 1)}-q_{k-1}{\bf 1}_{(k\ge 2)}-q_{k+1}.
$
By Dynkin's formula, the compensated process
\begin{eqnarray*}
M(t) & := & Q(t)-Q(0) -\int_0^t D\left(Q(s)\right)ds
\end{eqnarray*}
is a $\ell^1(\dZ_+)-$valued martingale.
Comparing with (\ref{cauchy}) and using (\ref{lipschitz})-(\ref{drift}), we easily obtain 
\begin{eqnarray*}
\left\|Q(t)-\qq(t)\right\| & \le & \varepsilon(t) + 2\int_0^t(2+\|\qq(s)\|) \left\|Q(s)-\qq(s)\right\| ds,
\end{eqnarray*}
for all $t<T_\star$, where we have set 
\begin{eqnarray*}
\varepsilon(t) & := & \left\|Q(0)-q\right\|  + \frac{1}{n}\left\|\int_0^tR\left(Q(s)\right)ds\right\|+ \left\|M(t)\right\|.
\end{eqnarray*}
We may now fix $0\le T < T_\star$ and apply Gr\"onwall's Lemma to obtain
\begin{eqnarray*}
\sup_{t\in[0,T]}\left\|Q(t)-\qq(t)\right\| & \le & \left(\sup_{t\in[0,T]} \varepsilon(t)\right)\exp\left\{2\int_0^T(2+\|\qq(s)\|) ds\right\}.
\end{eqnarray*}
In order to establish the claim for $T<T_\star$, it therefore suffices to show that 
\begin{eqnarray}
\label{liquid:toshow}
\sup_{t\in[0,T]} \varepsilon(t) & \xrightarrow[n\to\infty]{\PP} & 0.
\end{eqnarray}
This will also guarantee that $\qq(t) \in \cP(\dZ_+)$ for all $t\in[0,T_\star)$, thereby ruling out the possibility that $\|\qq(t)\|$ explodes in finite time. We will thus have $T_\star  =  \infty$, and the proof will be complete. To prove (\ref{liquid:toshow}), we treat each term appearing in the definition of $\varepsilon(t)$ separately. The first one vanishes by (\ref{assume:empirical}). For the second, we observe that $\|R(q)\|\le 4\|q\|$ for all $q\in\ell^1(\dZ_+)$, so that
\begin{eqnarray*}
\sup_{t\in[0,T]}\frac{1}{n}\left\|\int_0^tR\left(Q(s)\right)ds\right\| & \le & \frac{4T}{n}.
\end{eqnarray*}
Finally, for the martingale term, we note that the $k$th coordinate $M_k$ is a continuous-time martingale with jumps of size at most $\frac{2}{n}$ occurring at rate at most $n\left(2Q_k(t)+Q_{k-1}(t)+Q_{k+1}(t)\right)dt$. Thus,
\begin{eqnarray*}
\EE\left[\left|M_k(T)\right|\right] & \le & 2\int_0^T\EE\left[2Q_k(t)+{\bf 1}_{(k\ge 1)}Q_{k-1}(t)+Q_{k+1}(t)\right]dt.\\
\EE\left[\left(M_k(T)\right)^2\right] & \le & \frac{4}{n}\int_0^T\EE\left[2Q_k(t)+{\bf 1}_{(k\ge 1)}Q_{k-1}(t)+Q_{k+1}(t)\right]dt.
\end{eqnarray*}
Since $\sum_k kQ_k(t)=\frac{m}{n}$, we deduce that in the regime (\ref{assume:sparse}),
\begin{eqnarray*}
\sum_{k=1}^\infty (k+1)\, \EE[|M_k(T)|] \ = \ \cO(1) & \textrm{ and } & \sum_{k=1}^\infty (k+1)\, \EE\left[\left(M_k(T)\right)^2\right] \ = \ \cO\left(\frac 1n\right). 
\end{eqnarray*}
This is more than enough to imply $\EE\left[\left\|M(T)\right\|\right]\to 0$. The convergence $\sup_{t\in[0,T]}\|M(t)\|\xrightarrow[]{\PP} 0$ then follows from Doob's maximal inequality applied to the sub-martingale $\left(\|M(t)\|\right)_{t\ge 0}$.

\subsection{Probabilistic representation of the fluid limit}
We now turn to the analysis of the fluid limit $(\qq(t)\colon t\ge 0)$. The latter trivializes in the degenerate case $\lambda=0$, and we will henceforth assume that $\lambda>0$. Let  $\Xi^+,\Xi^-$ be two independent Poisson point processes with unit intensity on $\dR_+$, and define a process $\bZ:=(Z(t)\colon t\ge 0)$ by the formula
\begin{eqnarray}
\label{def:w}
Z(t) & := & \Xi^+\left(\int_0^t(1-\qq_0(s))ds\right)-\Xi^-\left(t\right).
\end{eqnarray}
Now, let $X(0)$ be a $q-$distributed variable independent of $\Xi^\pm$, and consider the reflected process
\begin{eqnarray}
\label{def:z}
X(t) & := & \left(X(0)+Z(t)\right)\vee \max_{s\in[0,t]}\left(Z(t)-Z(s)\right).
\end{eqnarray}
In words, $\bX=\left(X(t)\colon t\ge 0\right)$ is a time-inhomogeneous birth-and-death process with initial law $q$, upward rate $1-\qq_0(t)$ and downward rate $1$. Comparing the associated Kolmogorov equations with (\ref{cauchy}), we see that $\bX$ ``represents'' our fluid limit $(\qq(t)\colon t\ge 0)$ in the sense that
 \begin{eqnarray}
\label{representation}
\qq_k(t) & = & \PP\left(X(t)=k\right),
\end{eqnarray}
for all $k\in\dZ_+$ and all $t\in\dR_+$. Note in particular that $1-\qq_0(t)=\PP(X(t)>0)$, so that $\bX$ can be autonomously described as a time-inhomogeneous birth-and-death process with downward rate $1$ and upward rate $\PP(X(t)>0)$. We now enumerate a few consequences of this representation. 
\begin{lemma}[Mixing for the fluid]\label{lm:mixing} $X(t)$ is asymptotically independent of $X(0)$, i.e. for all $k,\ell\ge 0$,
\begin{eqnarray*}
\PP\left(X(0)=k,X(t)=\ell\right)-\PP\left(X(0)=k\right)\PP\left(X(t)=\ell\right) & \xrightarrow[t\to\infty]{} & 0.
\end{eqnarray*}
\end{lemma}
\begin{proof}
For each $k\in\dZ_+$, define a process $\bX^k$ by
\begin{eqnarray*}
X^k(t) & := & \left(k+Z(t)\right)\vee \max_{s\in[0,t]}\left(Z(t)-Z(s)\right),
\end{eqnarray*}
so that $\bX^k$ coincide with $\bX$ on the event $\{X(0)=k\}$. Since $X(0)$ is independent of $\bZ$, we have
\begin{eqnarray*}
\PP\left(X(0)=k,X(t)=\ell\right) & = & \PP(X(0)=k)\PP(X^k(t)=\ell).
\end{eqnarray*}
By construction, we have $X^k(t)=X^0(t)$ for all $t\ge 
T_k  :=  \inf\{t\ge 0\colon Z(t)=-k\}$. Consequently, 
\begin{eqnarray*}
\left |\PP\left(X^k(t)=\ell\right)-\PP\left(X(t)=\ell\right)\right | & \le & \sum_{i=0}^\infty\PP\left(X(0)=i\right)\PP\left(T_{i\vee k}\ge t\right).
\end{eqnarray*}
The conclusion now follows by letting $t\to\infty$ and observing that the $T_k's$ are almost-surely finite, since $\bZ$ has upward rate at most $1$ and downward rate $1$.
\end{proof}

\begin{lemma}[Conservation of mass]\label{lm:mass} 
$\EE[X(t)] = \lambda$ for all $t\ge 0$. 
\end{lemma}
\begin{proof}
For any bounded observable $\psi\colon \dZ_+\to\dR$ and any time $t\ge 0$, Dynkin's formula ensures that 
\begin{eqnarray*}
\EE\left[\psi(X_t)\right] &  = & \EE\left[\psi(X_0)\right] + \int_0^t\EE\left[\left(1-\qq_0(u)\right)\Delta\psi(X_u)-\Delta\psi(X_u-1)\right]du,
\end{eqnarray*}
where $\Delta\psi(-1)=0$ and $\Delta\psi(x)=\psi(x+1)-\psi(x)$ for $x\in\dZ_+$. By monotone convergence, the formula extends to the case $\psi(x)=x$. But then $\Delta\psi(x)={\bf 1}_{(x\ge 0)}$, and the integral vanishes. 
\end{proof}

\begin{lemma}[Lower-bound on void probability]
\label{lm:void}
For each $s>0$, we have
\begin{eqnarray*}
\inf_{t\ge 0}\,\PP\left(X(t+s)=0\right) & > & 0.
\end{eqnarray*} 
\begin{proof}
Let $Y=(Y(u)\colon u\ge 0)$ take the  value $\lfloor 2\lambda\rfloor$ over the time interval $[0,t]$ and then evolve as a simple random walk on $\dZ_+$ from time $t$ onwards (i.e., it jumps up and down at unit rate, except that jumps from $0$ to $-1$ are censored). Since our original process $X$ has the same downward rates and  lower upward rates, we may couple $X$ and $Y$ in such a way that
\begin{enumerate}[(i)]
\item $\left(X(u)\colon u\in[0,t]\right)$ is independent of $Y$; 
\item from time $t$ onwards, the attempts to jump downwards occur at the same times for $X$ and $Y$;
\item from time $t$ onwards, whenever $X$ jumps upwards, so does $Y$.
\end{enumerate} 
Properties (ii)-(iii) guarantee the inclusion
$
\left\{X(t)\le \lfloor 2\lambda\rfloor\right\} \subseteq \left\{X(s+t)\le {Y}(s+t)\right\}$. 
In particular,
\begin{eqnarray*}
\left\{X(s+t)=0\right\}
& \supseteq &
\left\{Y(s+t)=0\right\}\cap\left\{X(t)\le \lfloor 2\lambda\rfloor\right\}. 
\end{eqnarray*}
By (i), the two events on the right-hand side are independent. The first has probability $\kappa_s(\lfloor 2\lambda\rfloor,0)>0$, where $\kappa$ denotes the transition kernel for simple random walk on $\dZ_+$. For the second, we may invoke Markov's inequality and lemma \ref{lm:mass} to write
\begin{eqnarray*}
\PP\left(X(t)\le \lfloor 2\lambda\rfloor\right) & \ge & 1-\frac{\lambda}{\lfloor 2\lambda\rfloor+1}.
\end{eqnarray*}
The right-hand side exceeds $\frac 12$, and we conclude that $\PP\left(X(t+s)=0\right)\ge \frac 12\kappa_s(\lfloor 2\lambda\rfloor,0)$.   
\end{proof}
\end{lemma}
\begin{lemma}[Uniform integrability]\label{lm:UI}
The process $\bX$ is uniformly integrable.
\end{lemma}
\begin{proof} 
On $[0,1]$, the representation (\ref{def:w}-\ref{def:z}) immediately yields the domination
\begin{eqnarray*}
\max\left\{X(t)\colon t\in[0,1]\right\} & \le & X(0)+\Xi^+(1),
\end{eqnarray*}
and the right-hand side has finite mean. On the other hand, on $[1,\infty)$, Lemma \ref{lm:void} guarantees that the upward jump rate $1-\qq_0$ is less than $1-\varepsilon$ for some $\varepsilon>0$. Consequently, we can couple $(X(t)\colon t\ge 1)$ with an homogeneous birth-and-death process $\bY=\left(Y(t)\colon t\ge 1\right)$ starting at zero and jumping up at rate $1-\varepsilon$ and down at rate $1$, in such a way  that
\begin{eqnarray*}
\forall t\in[1,\infty),\qquad X(t) & \le & X(1)+Y(t)
\end{eqnarray*}
Surely, starting $\bY$ from its stationary law $\cG(\frac{1}{\varepsilon})$ instead of $0$ can only make it larger, and hence $Y(t)$ is stochastically dominated by $\cG(\frac{1}{\varepsilon})$. In conclusion, $X(t)$ is stochastically dominated by the sum of three integrable  variables whose laws do not depend on $t$, and the claim is proved. 
\end{proof}
\subsection{Entropic relaxation}
Entropy will play a crucial role, see \cite{MR2239987} for an account. Recall that the \emph{entropy} of $p\in\cP(\dZ_+)$ is  
\begin{eqnarray*}
H(p) & := & \sum_{k=0}^\infty p_k\log\frac{1}{p_k} \ \in \ [0,\infty],
\end{eqnarray*}
with the convention $0\log \frac 10=0$, and where $\log$ denotes the natural logarithm. In particular,
\begin{eqnarray*}
H\left(\cG(\lambda)\right) & = & (1+\lambda)\log(1+\lambda) - \lambda\log {\lambda}.
\end{eqnarray*}
In fact, $\cG(\lambda)$ achieves the maximum entropy over all laws $p\in\cP(\dZ_+)$ with mean $\lambda$. Indeed, using the fact that $\log\left(\cG_k(\lambda)\right)$ is an affine function of $k$, it is straightforward to check that
\begin{eqnarray}
\label{entropymax}
H\left(\cG(\lambda)\right) & = & H(p)+D_{\textsc{kl}}\left(p\,\|\,\cG(\lambda)\right),
\end{eqnarray}
where $D_{\textsc{kl}}(p\,\|\,q)$ is the Kullback-Leibler divergence of $p$ w.r.t. a fully-supported law $q\in\cP(\dZ_+)$:
\begin{eqnarray}
\label{dkl}
D_{\textsc{kl}}(p\,\|\,q) & := & \sum_{k=0}^\infty q_k\,\phi\left(\frac{p_k}{q_k}\right),\qquad\textrm{with}\qquad \phi(u)=u\log u-(u-1)\ge 0.
\end{eqnarray}
Note that by strict convexity of $\phi$, we have $D_{\textsc{kl}}(p\,\|\,q)> 0$ unless $p=q$. Now, given a fully-supported law $p$ on $\dZ_+$, we define a quantity $V(p)\in[0,\infty]$ by
\begin{eqnarray*}
V(p) & := & (1-p_0)\left(D_{\textsc{kl}}(p\,\|\,\widehat{p})+D_{\textsc{kl}}(\widehat{p}\,\|\,p)\right),
\end{eqnarray*}
where the (fully-supported) law $\widehat{p}\in\cP(\dZ_+)$ is defined as follows: for all $k\in\dZ_+$,
\begin{eqnarray}
\widehat{p}_{k} & := & \frac{p_{k+1}}{1-p_0}.
\end{eqnarray}
Note that the geometric distributions are characterized by the \emph{memoryless property} $\widehat{p}=p$. In particular, $V(p)=0$ if only if $p$ is geometric, and this quantity may thus be viewed as measuring how far $p$ is from being geometric. The essence of Proposition \ref{pr:fluid} lies in the following  identity.  
\begin{lemma}[Entropy production]\label{lm:entropy} For all $t\ge 0$, we have
\begin{eqnarray*}
H\left(\qq(t)\right) & = & H\left(q\right)+\int_0^t V(\qq(u))\, du.
\end{eqnarray*}
\end{lemma}
\begin{proof}
Note that $\qq(t)$ has full support as soon as $t>0$, by our probabilistic representation (\ref{representation}) and the fact that the Poisson distribution has full support. Thus, the above integral is well defined, albeit possibly infinite at this stage. Now, the fluid equation (\ref{cauchy}) may be rewritten as follows: 
\begin{eqnarray*}
\frac{d\qq_k}{dt}\  = \ \mm_{k-1}-\mm_k, & \qquad\textrm{ where }\qquad  & \mm_k   :=   (1-\qq_0)\left(\qq_k-\widehat{\qq}_k\right)
\end{eqnarray*}
with the convention that $\mm_{-1}=0$. In particular, for $t\in\dR_+$ and $k\in\dZ_+$, we have
\begin{eqnarray*}
\qq_k(t)\log\frac{1}{\qq_k(t)} & = & q_k\log\frac{1}{q_k} + \int_{0}^t\left(\mm_{k}(u)-\mm_{k-1}(u)\right)\left(1+\log \qq_k(u)\right)du.
\end{eqnarray*}
Summing over $k$ and rearranging, we see that for all $K\ge 1$,
\begin{eqnarray}
\label{fluid:main}
\sum_{k=0}^K \qq_k(t)\log\frac{1}{\qq_k(t)} & = & \sum_{k=0}^K q_k\log\frac{1}{q_k}+\int_{0}^t{\mathfrak v}_K(u)\,du+\int_0^t\varepsilon_K(u)du,
\end{eqnarray}
where we have set
\begin{eqnarray*}
{\mathfrak v}_K & := & \sum_{k=0}^{K-1}\mm_{k}\log \frac{\qq_{k}}{\widehat{\qq}_{k}}  \ = \ (1-\qq_0)\sum_{k=0}^{K-1}\left(\qq_k\,\phi\left(\frac{\qq_k}{\widehat{\qq}_k}\right)+\widehat{\qq}_k\,\phi\left(\frac{\widehat{\qq}_k}{\qq_k}\right)\right),\\
\varepsilon_K & := & \mm_K(1+\log \qq_K)+\left(\sum_{k=0}^{K-1}\mm_k\right)\log\frac{1}{1-\qq_0}.
\end{eqnarray*}
Since ${\mathfrak v}_K\uparrow V(\qq)$ as $K\to\infty$, the claim will readily follow from (\ref{fluid:main}), provided we can show that 
\begin{eqnarray}
\label{fluid:toshow}
\int_0^t\varepsilon_K(u)du & \xrightarrow[K\to\infty]{} & 0,
\end{eqnarray}
which we now do. First, Lemma \ref{lm:mass} ensures that  the series $\sum_k \qq_k$ converges uniformly on $\dR_+$.
Note also that $|\mm_k|\le \qq_k+\qq_{k+1}$ and that $\sum_k\mm_k=0$. From this, it follows that 
\begin{eqnarray*}
\mm_K+(1-\qq_0)\qq_K\log \qq_K+ \left(\sum_{k=0}^{K-1}\mm_k\right)\log\frac{1}{1-\qq_0} 
& \xrightarrow[K\to\infty]{} &  0,
\end{eqnarray*}
uniformly on compact sets. Comparing with the definition of $\varepsilon_K$, we see that as $K\to\infty$,
\begin{eqnarray}
\label{fluid:weak}
\int_0^t\varepsilon_K(u)du & = & \int_0^t\qq_{K+1}(u)\log \frac 1{\qq_K(u)}du + o(1).
\end{eqnarray}
Note that $\qq_{K+1}\log \frac 1{\qq_K}\ge 0$. We may therefore pass to the limit in (\ref{fluid:main}) to obtain the inequality
\begin{eqnarray*}
H(\qq(t)) & \ge & H(q)+\int_0^t V(\qq(u))\, du.
\end{eqnarray*}
In particular, the integral on the right-hand side must be finite. By definition of $V$, this implies 
\begin{eqnarray*}
\int_0^t \qq_{K+1}(u)\log\frac{\widehat{\qq}_{K}(u)}{\qq_K(u)}\, du & \xrightarrow[K\to\infty]{} & 0.
\end{eqnarray*}
In view of (\ref{fluid:weak}) and the uniform convergence ${\qq}_{K+1}\log {\widehat{\qq}}_{K}\to 0$, we now readily obtain (\ref{fluid:toshow}).
\end{proof}

\begin{proof}[Proof of Proposition \ref{pr:fluid}]
By Pinsker's inequality, we have for all $t\ge 0$,
\begin{eqnarray}
\label{pinsker}
\frac 12\|\qq(t)-\cG(\lambda)\|^2 & \le & D_{\textsc{kl}}\left(\qq(t)\,\|\,\cG(\lambda)\right) \ = \ H\left(\cG(\lambda)\right) - H(\qq(t)),
\end{eqnarray}
where the equality follows from Lemma \ref{lm:mass} and the observation (\ref{entropymax}).
Now, the limit 
\begin{eqnarray*}
H_{\infty} & := & \lim_{t\to\infty}\uparrow H(\qq(t))
\end{eqnarray*} exists by Lemma \ref{lm:entropy}, and so our proof boils down to showing that $H_\infty\ge H\left(\cG(\lambda)\right)$. By Fatou's Lemma, it suffices to exhibit a sequence $(t_n)_{n\ge 1}$ along which 
$
\qq(t_n) \to \cG(\lambda).
$
To do so, observe that Lemma \ref{lm:entropy} forces
$
\inf_{t\ge 0}V(\qq(t))  = 0,
$
as otherwise $H(\qq(t))$ would diverge as $t\to\infty$, violating (\ref{pinsker}). We can thus find a sequence of times $(t_n)_{n\ge 1}$ along which
\begin{eqnarray}
\label{V0}
V(\qq(t_n)) & \xrightarrow[n\to\infty]{} & 0.
\end{eqnarray}
On the other hand, by Lemma \ref{lm:UI}, the collection $\left(\qq(t)\colon t\ge 0\right)$ is relatively compact w.r.t. the $1$-Wasserstein metric. We can thus assume (upon further extraction) that $\qq(t_n)\to p$, with $p\in\cP(\dZ_+)$ having mean $\lambda$. It then follows from (\ref{V0}) that  $p=\widehat{p}$, and therefore $p=\cG(\lambda)$, as desired.
\end{proof}

\section{Fast mixing in the absence of a solid phase}
\label{sec:fast}
In this section, we establish the special case $u_1=0$ of Theorem \ref{th:main}, as stated in Proposition \ref{pr:fast}. To do so, we deal with each of the issues enumerated below Proposition \ref{pr:fast}, in order of appearance. 
\subsection{Uniform downward drift}

To deal with issue $(i)$, we show that, starting from any state $\eta$, the  uniform integrability of $Q(t)$ is guaranteed after a time $t=\Theta\left(\max\eta\right)$ only. This is contained in the following Proposition, which  asserts that the number of particles on any non-empty site decreases at a linear rate. The uniformity in $n$ comes from the fact that, in the regime  (\ref{assume:sparse}), the density of particles per site is at most a constant $\rho_\star$ that does not depends on $n$: 
\begin{eqnarray}
\label{assume:bounded}
\frac{m}{n} & \le & \rho_\star.
\end{eqnarray}
\begin{proposition}[Uniform downward drift]\label{pr:uniform}
There are constants $\theta,\delta>0$, depending on $\rho_\star$ only, such that for any $n\ge 2$, any initial state $\eta\in\Omega$,  any $i\in\{1,\ldots,n\}$, and any time $t\in\dR_+$,
\begin{eqnarray*}
\EE\left[e^{\theta\eta_i(t) }\right] & \le & 2\left(1+e^{\theta\eta_i-\delta t}\right).
\end{eqnarray*}

\end{proposition}
The reason behind this result is the existence of a uniform lower-bound on the proportion of empty sites in the system after time $t=1$ ($1$ can actually be replaced by any positive constant).
\begin{lemma}[Many empty sites]\label{lm:uniform}There is a constant $\gamma\in(0,1)$, depending on $\rho_\star$ only, such that for any $n\ge 2$,  any initial state $\eta\in\Omega$ and any time $t\in[1,\infty)$,
\begin{eqnarray*}
\PP\left(Q_0(t)\le \gamma\right) & \le & e^{-\gamma n}.
\end{eqnarray*}
\end{lemma}
\begin{proof}
We can construct the zero-range process using an independent, rate$-\frac 1n$ Poisson point process $\Xi_{i\to j}$ for each pair $(i,j)\in[n]\times [n]$: the successive points of $\Xi_{i\to j}$ indicate the times at which site $i$ attempts to send a particle out to site $j$, and the move is allowed if and only if $i$ is not empty. Because of (\ref{assume:bounded}), at least half of the sites $i\in[n]$ must satisfy $\eta_i \le 2\rho_\star$, and we may thus select a subset $A$ of them with $|A|=\lceil n/2\rceil$. Note that $n-|A|=\lfloor n/2\rfloor\ge n/3$, since $n\ge 2$. For each $i\in A$, consider the "good" event 
\begin{eqnarray*}
G_i & := & \left\{\sum_{j\in [n]\setminus A}\Xi_{i\to j}\left([0,1]\right)\ge 2\rho_\star \right\}\bigcap \left\{\sum_{j\in[n]} \Xi_{j\to i}\left([0,1]\right)=0\right\}.
\end{eqnarray*}
Then by construction, we have $G_i\subseteq \left\{\eta_i(1)=0\right\}$, and hence
\begin{eqnarray*}
Q_0(1) & \ge & \frac{1}{n}\sum_{i\in A}{\bf 1}_{G_i}. 
\end{eqnarray*}
Since $\sum_{j\in[n]}\Xi_{j\to i}([0,1])$ and $\sum_{j\in [n]\setminus A}\Xi_{i\to j}([0,1])$ are independent Poisson random variables with mean $1$ and at least $1/3$ respectively, we have
\begin{eqnarray*}
\PP\left(G_i\right) & \ge & e^{-4/3}\sum_{k\ge 2\rho_\star}\frac{3^{-k}}{k!} \ =: \ p.
\end{eqnarray*}
Moreover, the events $(G_i)_{i\in A}$ are independent because $G_i$ depends only on the $\Xi_{k\to \ell}$ for $(k,\ell)$ in $$B_i:=(\{i\}\times [n]\setminus A)\cup([n]\times \{i\}),$$ and the $(B_i)_{i\in A}$ are pairwise disjoint. Thus, $ \frac{1}{n}\sum_{i\in A}{\bf 1}_{G_i}$ stochastically dominates a Binomial random variable with parameters $\lceil n/2\rceil$ and $p$. By Hoeffding's inequality, we deduce that
\begin{eqnarray*}
\PP\left(\frac{1}{n}\sum_{i\in A}{\bf 1}_{G_i}\le \frac{p}{4} \right) & \le & \exp\left(-\frac{p^2n}{4}\right),
\end{eqnarray*}
and so we may take $\gamma=p^2/4$ to obtain the claim for $t=1$. Since the result is uniform in the choice of the initial state $\eta$, the claim for $t\ge 1$ follows automatically by Markov's property. 
\end{proof}

\begin{proof}[Proof of Proposition \ref{pr:uniform}]
For $\theta>0$, Dynkin's formula ensures that $\phi_\theta(t):=\EE\left[e^{\theta \eta_i(t)}\right]$ satisfies
\begin{eqnarray}
\label{dynkin:uniform}
\frac{d \phi_\theta(t)}{dt} & = &   (e^\theta-1)\EE\left[e^{\theta \eta_i(t)}\left(1-Q_0(t)-\frac{1+e^{-\theta}(n-1)}{n}{\bf 1}_{(\eta_i(t)\ge 1)}\right)\right].
\end{eqnarray}
The trivial observation that the right-hand side is at most $(e^\theta-1)\phi_\theta(t)$ already yields
\begin{eqnarray}
\label{burnin}
\phi_\theta(t) & \le & \phi_\theta(0)\exp\left\{\left(e^\theta-1\right)t\right\}.
\end{eqnarray}
We will use this crude bound only for $t\in[0,1]$. For $t\ge 1$, we may instead invoke Lemma \ref{lm:uniform} to get
$
 \PP\left(Q_0(t)\le \gamma\right)  \le  e^{-\gamma n},
$
for some $\gamma>0$ that depends only on $\rho_\star$. Going back to (\ref{dynkin:uniform}), we have
\begin{eqnarray*}
\frac{d \phi_\theta(t)}{dt}  & \le & (e^\theta-1)\EE\left[e^{\theta \eta_i(t)}\left(1-Q_0(t)-e^{-\theta}\right)+e^{-\theta}\right] \\
& \le & (e^\theta-1)\left\{(1-\gamma-e^{-\theta})\phi_\theta(t)+e^{-\theta}+(1-e^{-\theta})e^{(\theta \rho_\star  - \gamma)n}\right\},
\end{eqnarray*}
where we have split the expectation according to whether $Q_0(t)> \gamma$ or $Q_0(t)\le \gamma$ and, in the latter case, used the crude bounds $Q_0(t)\ge 0$ and $\eta_i(t)\le \rho_\star n$. Let us now choose
$
\theta := \frac{\gamma}{1+\rho_\star},
$
so that $\theta\rho_\star-\gamma\le 0$   and $1-\gamma-e^{-\theta}\le -\theta \rho_\star$. We are then left with the differential inequality
\begin{eqnarray*}
\frac{d \phi_\theta(t)}{dt} & \le & (e^\theta-1)\left(1-\rho_\star\theta\phi_\theta(t)\right),
\end{eqnarray*}
which we may integrate to deduce that for all $t\ge 1$,
\begin{eqnarray*}
\phi_\theta(t) & \le & \frac{1}{\theta \rho_\star}+\left(\phi_\theta(1)-\frac{1}{\theta \rho_\star}\right)e^{-\delta (t-1)}
\end{eqnarray*}
where $\delta:=(e^{\theta}-1)\theta\rho_\star>0$. Combining this with (\ref{burnin}), we conclude that for all $t\ge 0$,
\begin{eqnarray*}
\phi_\theta(t) & \le & \kappa\left(1+\phi_\theta(0)e^{-\delta t}\right),
\end{eqnarray*}
where   $\kappa,\delta,\alpha$ depend only on $\rho$. Finally, observe that these three constants may  respectively be replaced with $\kappa^\alpha,\delta\alpha,\theta\alpha$ for any $\alpha\in(0,1)$, since by Jensen's inequality, 
\begin{eqnarray*}
\phi_{\theta\alpha}(t) & \le  & \left(\phi_{\theta}(t)\right)^\alpha \ \le \ \kappa^\alpha\left(1+\phi_{\theta\alpha}(0)e^{-\delta\alpha t}\right).
\end{eqnarray*}
Choosing $\alpha$ small enough will make $\kappa^\alpha\le 2$, and the result is proved.
\end{proof}

\subsection{Partial exchangeability}
To deal with issue (ii), we introduce an object that refines the empirical distribution $Q(t)$ studied in Section \ref{sec:liquid}: the \emph{empirical transition matrix} of the system,
\begin{eqnarray}
W(t) & := & \frac{1}{n}\sum_{i=1}^n\delta_{\left(\eta_i(0),\eta_i(t)\right)}.
\end{eqnarray}
In  words, for each $(k,\ell)\in\dZ_+^2$,  $W_{k,\ell}(t)$ is a $[0,1]-$random variable indicating the proportion of sites that start with $k$ particles at time $0$ and end up with $\ell$ particles at time $t$. 
Contrarily to $Q(t)$, the understanding of $W(t)$ suffices to fully recover the law of $\eta(t)$:
\begin{lemma}[Partial exchangeability]
\label{lm:exchange}
Fix an initial configuration $\eta$ and a time $t\ge 0$. Then the conditional law of $\eta(t)$ given $W(t)$ is uniform over all configurations $\xi\in\Omega$ such that
\begin{eqnarray}
\label{exchange}
\frac{1}{n}\sum_{i=1}^n\delta_{\left(\eta_i,\xi_i\right)} & = & W(t).
\end{eqnarray}
\end{lemma}
\begin{proof}
Since the rate at which a site attempts to send a particle to another site is the same for all pairs of sites, the zero-range process enjoys the following  obvious symmetry: if $(\eta(t)\colon t\ge 0)$ is a zero-range process and if $\sigma\colon [n]\to[n]$ is a permutation, then the process $(\eta'(t)\colon t\ge 0)$ defined by 
\begin{eqnarray*}
\eta'_i(t) &:=& \eta_{\sigma(i)}(t)
\end{eqnarray*} is again a zero-range process. In particular, if $\sigma$ preserves the initial state ($\eta'(0)=\eta(0)$), then the two processes have the same law. Since $W$ is invariant by such permutations, the result follows.
\end{proof}
In light of this, our task boils down to understanding the behavior of the process $\left(W(t)\colon t\ge 0\right)$. The following proposition lifts the results obtained for $Q(t)$ to $W(t)$. For convergence purposes, we regard arrays as elements of the Banach space $\ell^1(\dZ_+^2)$, with norm
$
\|w\|  =  \sum_{k,\ell}|w_{k,\ell}|.
$
\begin{proposition}[Matrix refinement]\label{pr:matrix} For any fixed time horizon $T$, in the regime  (\ref{assume:sparse})-(\ref{assume:empirical}), we have 
\begin{eqnarray*}
\sup_{t\in[0,T]}\left\|W(t)-w(t)\right\|  \xrightarrow[n\to\infty]{\PP}  0,
\end{eqnarray*}
where $w_{k,\ell}(t) = \PP(X(0)=k,X(t)=\ell)$ and $(X(t)\colon t\ge 0)$ is the process defined at (\ref{def:z}). 
\end{proposition}
\begin{proof} 
The proof mimics that of Proposition \ref{pr:chaos}, except that the existence of the fluid limit is here already granted. For a time-inhomogeneous birth-and-death chain $\bX$ with downward rate $1$ and upward rate $r(t)$, the law $w(t)$ of the pair $(X(0),X(t))$ satisfies the differential equation
\begin{eqnarray*}
\frac{dw_{k,\ell}}{dt}  & := & w_{k,\ell+1}- {\bf 1}_{\left(\ell\ge 1\right)}w_{k,\ell}  - r\left(w_{k,\ell}-{\bf 1}_{\left(\ell\ge 1\right)}w_{k,\ell-1}\right), \qquad (k,\ell)\in\dZ_+^2.
\end{eqnarray*}
Here, we further have $r(t)=\PP(X(t)\ge 1)=\sum_{k\ge 0}\sum_{\ell\ge 1}w_{k,\ell}(t).$ Consequently, for all $t\in\dR_+$,
\begin{eqnarray}
\label{matrix:fluid}
w(t) & = & w(0)+\int_0^tF(w(s))ds,
\end{eqnarray}
where the drift $F\colon \ell^1(\dZ_+^2)\to \ell^1(\dZ_+)$ is defined by 
\begin{eqnarray*}
F_{k,\ell}(w) & = & w_{k,\ell+1}-{\bf 1}_{\left(\ell\ge 1\right)}w_{k,\ell}-\left(\sum_{k\ge 0}\sum_{\ell\ge 1}w_{k,\ell}\right)\left(w_{k,\ell}-{\bf 1}_{\left(\ell\ge 1\right)}w_{k,\ell-1}\right).
\end{eqnarray*}
Observe that $F$ is locally Lipschitz: for $w,w'\in\ell^1(\dZ_+^2)$, 
\begin{eqnarray}
\label{matrix:lip}
\left\|F(w)-F(w')\right\| & \le & 2\left(1+\|w\|+\|w'\|\right)\|w-w'\|.
\end{eqnarray}
On the other hand,  $(W(t)\colon t\ge 0)$ is a Markov process on a finite subset of $\cP(\dZ^2_+)$ with jumps  $w\mapsto w+\frac{1}{n}\Delta_{i,j,k,\ell}$ occurring at rate $c_{i,j,k,\ell}(w)$, where for each $(i,j,k,\ell)\in\dZ_+^4$, 
\begin{eqnarray*}
\Delta_{i,j,k,\ell} & = & \delta_{i,j-1}+\delta_{k,\ell+1}-\delta_{i,j}-\delta_{k,\ell}\\
c_{i,j,k,\ell}(w) & = & {\bf 1}_{(j\ge 1)}w_{i,j}\left(n w_{k,\ell}-{\bf 1}_{((i,j)=(k,\ell))}\right).
\end{eqnarray*}
Consequently, Dynkin's formula asserts that the compensated process
\begin{eqnarray}
\label{matrix:dynkin}
M(t) & := & W(t)-W(0)  -\int_0^t D\left(W(s)\right)ds
\end{eqnarray}
is a  $\ell^1(\dZ_+^2)-$valued martingale, where the infinitesimal drift $w\mapsto D(w)$ is given by
\begin{eqnarray*}
D(w) & := & \frac{1}{n}\sum_{(i,j,k,l)\in\dZ_+^4}c_{i,j,k,\ell}(w){\Delta_{i,j,k,\ell}}.
\end{eqnarray*}
Comparing with the definition of $F$, we see that $D=F+\frac{1}{n}R$, where 
\begin{eqnarray*}
R_{k,\ell}(w) & := & w_{k,\ell+1}+w_{k,\ell-1}{\bf 1}_{(\ell\ge 2)}-2w_{k,\ell}{\bf 1}_{(\ell\ge 1)}.
\end{eqnarray*}
Subtracting (\ref{matrix:fluid}) from (\ref{matrix:dynkin})  and using (\ref{matrix:lip}), we obtain
\begin{eqnarray*}
\left\|W(t)-w(t)\right\| & \le & \varepsilon(t)+6\int_0^t\left\|W(s)-w(s)\right\|ds,
\end{eqnarray*}
where we have set
\begin{eqnarray*}
\varepsilon(t) & := & \|W(0)-w(0)\|+\frac{1}{n}\left\|\int_0^tR\left(W(s)\right)ds\right\|+\|M(t)\|.
\end{eqnarray*}
By Gr\"onwall's Lemma, we deduce that 
\begin{eqnarray*}
\sup_{t\in[0,T]}\left\|W(t)-w(t)\right\| & \le & \left(\sup_{t\in[0,T]}\left\|\varepsilon(t)\right\|\right)e^{6T},
\end{eqnarray*}
and it only remains to show that $\varepsilon(t)\xrightarrow[]{\PP} 0$ as $n\to\infty$. We treat each term appearing in the definition of $\varepsilon(t)$ separately. The first one vanishes by assumption (\ref{assume:empirical}). For the second one, it suffices to note that $\|R(w)\|\le 4\|w\|$, so that
\begin{eqnarray*}
\sup_{t\in[0,T]}\frac{1}{n}\left\|\int_0^t R\left(W(s)\right)ds\right\| & \le & \frac{4T}{n}.
\end{eqnarray*}
Finally, the convergence $\sup_{t\in[0,T]}\left\|M(t)\right\|\xrightarrow[]{\PP} 0$ will follow from Doob's maximal inequality if we can show that $\EE\left[\|M(T)\|\right]\to 0$. For each fixed $(k,\ell)\in\dZ_+^2$, $M_{k,\ell}$ is a real-valued martingale with jumps of size at most $\frac{2}n$ and jump rate at most $n\left(2W_{k,\ell}+W_{k,{\ell+1}}+W_{k,{\ell-1}}{\bf 1}_{(\ell\ge 1)}\right)$, so
\begin{eqnarray*}
\EE\left[\left|M_{k,\ell}(T)\right|\right] & \le & 2\int_0^T\EE\left[2W_{k,\ell}(t)+W_{k,{\ell+1}}(t)+W_{k,{\ell-1}}(t){\bf 1}_{(\ell\ge 1)}\right]dt.
\\
\EE\left[\left(M_{k,\ell}(T)\right)^2\right] & \le & \frac{4}{n}\int_0^T\EE\left[2W_{k,\ell}(t)+W_{k,{\ell+1}}(t)+W_{k,{\ell-1}}(t){\bf 1}_{(\ell\ge 1)}\right]dt.
\end{eqnarray*}
Since $\sum_{k,\ell}(k+\ell)W_{i,j}(T)=\frac{2m}{n}$, we deduce that in the regime (\ref{assume:sparse}),
\begin{eqnarray*}
\sum_{k,\ell}(k+\ell+1)\, \EE[|M_{k,\ell}(T)|] \ = \ \cO(1) & \textrm{ and } & \sum_{k,\ell}(k+\ell+1)\, \EE\left[\left(M_{k,\ell}(T)\right)^2\right] \ = \ \cO\left(\frac 1n\right). 
\end{eqnarray*}
This is more than enough to ensure that $\EE\left[\left\|M(T)\right\|\right]\to 0$, as desired.
\end{proof}

\subsection{Spectral gap argument}
To deal with issue (iii), we exploit a \emph{spectral gap} contraction argument.
Consider an irreducible, continuous-time Markov process on a finite state space $\Omega$, with generator $\cL$ and stationary law $\pi$. If $\pi$ is reversible, then $-\cL$ is a non-negative self-adjoint operator  on the Hilbert space $\ell^2(\pi)$, and the spectral gap is defined as its smallest non-zero eigenvalue:
\begin{eqnarray}
\gap & := & \min\left\{\lambda>0\colon \textrm{ker}(\lambda +\cL)\ne\emptyset\right\}.
\end{eqnarray}
 This fundamental parameter can be used to bound the total-variation distance to equilibrium via the following classical inequality (see, e.g., \cite{MR2341319}): for any initial law $\nu\in\cP(\Omega)$ and any time $t\in\dR_+$,
\begin{eqnarray}
\label{gap:contraction}
\left\|\nu P_t-\pi\right\|_{\textsc{tv}} & \le & \frac{1}{2}\left(\max_{x\in\Omega}\frac{\nu(x)}{\pi(x)}\right)^{1/2}{e^{-\gap\, t}}.
\end{eqnarray}
In the case of the mean-field zero-range process, the spectral gap was estimated by Morris \cite{MR2271475}. 
\begin{theorem}[Morris]\label{th:morris}
In the regime (\ref{assume:sparse}), the spectral gap is bounded away from $0$, i.e. ${\gap} = \Omega(1).$
\end{theorem}
Thus, the right-hand side of (\ref{gap:contraction}) decreases exponentially fast with $t$. Since the size of the state space $|\Omega|$ grows exponentially in $n$, maximizing over $\nu$ in (\ref{gap:contraction}) leaves us with the worst-case bound $
\tmix(\varepsilon) =  \cO(n),
$
which has the right order of magnitude, but is rather remote from our current aim: we want to prove mixing in time $o(n)$ in the absence of a solid phase. In that case, Proposition \ref{pr:matrix} will be shown to imply that the
relative entropy to equilibrium,
\begin{eqnarray*}
D_{\textsc{kl}}\left(\nu \,\|\,\pi\right) & := & \sum_{x\in\Omega}\nu(x)\log\frac{\nu(x)}{\pi(x)},
\end{eqnarray*}
quickly becomes $o(n)$. Once there, the following lemma will be invoked to conclude. 
\begin{lemma}[Fast mixing once relative entropy is small]\label{lm:entropygap}Consider a continuous-time Markov chain with reversible law $\pi$ on a finite space $\Omega$. Fix an initial law $\nu\in\cP(\Omega)$ and  $\varepsilon\in(0,1)$, and set
\begin{eqnarray}
\label{ineq:entropygap}
t & := & \frac{1}{\gap}\left(\frac{D_{\textsc{kl}}\left(\nu||\pi\right)}{\varepsilon}+\log\left(\frac{1}{\varepsilon}\right)+1\right).
\end{eqnarray}
Then,  $\|\nu P_t-\pi\|_{\textsc{tv}}\le \varepsilon$, where $P_t$ denotes the transition kernel of the process.
\end{lemma}
\begin{proof}
Consider the subset $\cS\subseteq \Omega$ defined by 
\begin{eqnarray*}
\cS:=\left\{x\in\Omega\colon \log\frac{\nu(x)}{\pi(x)} \le 1+\frac{2D_{\textsc{kl}}\left(\nu||\pi\right)}{\varepsilon}\right\}.
\end{eqnarray*}
Observe that by definition,
\begin{eqnarray*}
\left(1+\frac{2D_{\textsc{kl}}\left(\nu||\pi\right)}{\varepsilon}\right) \nu(\cS^c) & \le & \sum_{x\in S^c}\nu(x)\log\frac{\nu(x)}{\pi(x)} \\& \le & D_{\textsc{kl}}\left(\nu||\pi\right)+\sum_{x\in\cS}\nu(x)\log\frac{\pi(x)}{\nu(x)} \\ & \le &  D_{\textsc{kl}}\left(\nu||\pi\right)+\pi(\cS)-\nu(\cS)\\
& \le  & D_{\textsc{kl}}\left(\nu||\pi\right) +\nu(\cS^c)
\end{eqnarray*}
where at the third line we have used $\log u\le u-1$. After simplification, we are left with 
\begin{eqnarray*}
\nu(\cS^c) & \le & \frac{\varepsilon}{2}.
\end{eqnarray*}
Now, let $\widehat{\nu}:=\nu(\cdot|\cS)$ denote the projection of $\nu$ onto $\cS$. Note that
\begin{eqnarray*}
\max_{x\in\Omega}\frac{\widehat{\nu}(x)}{\pi(x)}& 
 = & \frac{1}{\nu(\cS)}\max_{x\in\cS}\frac{\nu(x)}{\pi(x)} \ \le \ \exp\left\{2+\frac{2D_{\textsc{kl}}\left(\nu||\pi\right)}{\varepsilon}\right\},
\end{eqnarray*}
because $\nu(\cS)\ge 1/2 \ge 1/e$. Consequently, (\ref{gap:contraction}) shows that for all $t\ge 0$, 
\begin{eqnarray*}
\left\|\widehat{\nu}P_t-\pi\right\|_{\textsc{tv}} & \le & \frac{1}{2}\exp\left\{1+\frac{D_{\textsc{kl}}\left(\nu||\pi\right)}{\varepsilon}-\gap\,{t}\right\}.
\end{eqnarray*}
Choosing $t$ as in (\ref{ineq:entropygap}) sets the right-hand side to $\varepsilon/2$. On the other hand, we trivially have
\begin{eqnarray*}
\|\widehat{\nu} P_t-\nu P_t\|_{\textsc{tv}} & \le & \|\widehat{\nu}-\nu \|_{\textsc{tv}}  \ = \ \nu(\cS^c) \ \le \ \frac{\varepsilon}{2}.
\end{eqnarray*}
By the triangle inequality, we deduce that $\|{\nu} P_t-\pi\|_{\textsc{tv}}\le \varepsilon$, as desired.
\end{proof}
\subsection{Proof of fast mixing}
We are now ready to establish Proposition \ref{pr:fast}. First, by Proposition \ref{pr:uniform}, there are constants $\delta,\theta>0$ that do not depend on $n$, such that for
\begin{eqnarray}
\label{timescale}
t & = & \frac{\theta}{\delta}\max_i\eta_i,
\end{eqnarray}
we have $\max_{i}\EE[e^{\theta\eta_i(t)}] \le 4$. By Markov's inequality, this implies that
\begin{eqnarray*}
\PP\left(\eta(t) \notin K_r\right) & \le & \frac{4}{r},
\end{eqnarray*}
where $K_r:=\left\{\xi\in\Omega \colon \frac{1}{n}\sum_{i=1}^ne^{\theta\xi_i} \le r\right\}$. On the other hand, by the Markov property,
\begin{eqnarray}
\|P_{t+s}(\eta,\cdot)-\pi\|_{\textsc{tv}} & \le & \PP\left(\eta(t)\notin K_r\right) + \max_{\xi\in K_r}\|P_{s}(\xi,\cdot)-\pi\|_{\textsc{tv}}.
\end{eqnarray}
Thus, Proposition \ref{pr:fast} will follow if we can show fast mixing from any configuration in $K_r$, where $r$ is allowed to be arbitrarily large but fixed independently of $n$. In words, the uniform downward drift allows us to replace the assumption (\ref{assume:liquid}) by the much stronger condition
\begin{eqnarray}
\frac{1}{n}\sum_{i=1}^ne^{\theta\eta_i} & = & \cO(1).
\end{eqnarray}
In this regime, the empirical distribution $\frac{1}{n}\sum_{i=1}^n\delta_{\eta_i}$ is uniformly integrable: upon passing to a subsequence, we may assume that (\ref{assume:empirical}) holds, with the limit $q$ having mean $\lambda=\rho$.
Under this condition, we will now show that for any $s=s(n)$ that diverges with $n$ (say, $s=\log n$),
\begin{eqnarray}
\label{conclude:entropy}
D_{\textsc{kl}}\left(P_{s}(\eta,\cdot)||\pi\right) & = & o(n).
\end{eqnarray}
The conclusion will then follow by applying Lemma \ref{lm:entropygap} with $\nu=P_s(\eta,\cdot)$: indeed, the time $t$ defined at (\ref{ineq:entropygap}) satisfies $t=o(n)$, and we have $\|P_{t+s}(\eta,\cdot)-\pi\|_{\textsc{tv}}=\|\nu P_{t}-\pi\|_{\textsc{tv}}\le \varepsilon$, showing that $\tmix(\eta;\varepsilon)\le t+s=o(n)$, as desired. The remainder of the section is devoted to proving (\ref{conclude:entropy}). 

Define the \emph{combinatorial entropy} of the non-negative integers $a_0,\ldots,a_K$ to be
\begin{eqnarray*}
h\left(a_0,\ldots,a_K\right) & := & \log\left\{{a_0+a_1+\cdots+a_K \choose a_0, a_1,\ldots,a_K }\right\},
\end{eqnarray*}
and extend this definition to finitely-supported sequences $a_0,a_1,\ldots$ by simply ignoring the non-zero entries. 
Now, consider one sequence $(a_0^n,a_1^n,\ldots)$ for each value of $n\ge 1$, and assume that 
\begin{eqnarray*}
& (i) & a_0^n+a_1^n+\cdots \ \xrightarrow[n\to\infty]{} + \infty  \\
& (ii) & \forall k\in \dZ_+, \ \frac{a_k^n}{a_0^n+a_1^n+\cdots}\xrightarrow[n\to\infty]{}  p_k,
\end{eqnarray*}
for some law $p\in\cP(\dZ_+)$. Then a classical application of Stirling's approximation yields
\begin{eqnarray}
\label{stirling}
\liminf_{n\to\infty}\frac{h\left(a_0^n,a_1^n,\ldots\right)}{a_0^n+a_1^n+\cdots} & \ge & H(p).
\end{eqnarray}
Now, observe that the number of configurations $\xi\in\Omega$ satisfying (\ref{exchange}) is precisely \begin{eqnarray*}
\exp\left(\sum_{i=0}^\infty h\left(nW_{k,\ell}(t)\colon \ell\in\dZ_+\right)\right).
\end{eqnarray*}
Using Lemma \ref{lm:exchange} and the classical fact that \emph{conditioning reduces entropy} (see \cite{MR2239987}),  we deduce that
\begin{eqnarray*}
H\left(P_t(\eta,\cdot)\right) & \ge & \sum_{i=0}^\infty \EE\left[h\left(nW_{k,\ell}(t)\colon \ell\in\dZ_+\right)\right].
\end{eqnarray*}
We may finally let $n\to\infty$: Proposition \ref{pr:matrix} ensures that  for fixed $t\ge 0$ and $(k,\ell)\in \dZ_+^2$, we have 
\begin{eqnarray*}
W_{k,\ell}(t) & \xrightarrow[n\to\infty]{\PP} & w_{k,\ell}(t) \ = \ \PP\left(X(0)=k,X(t)=\ell\right).
\end{eqnarray*}
Applying (\ref{stirling}) with $a_\ell^n=nW_{k,\ell}(t)$ and then Fatou's lemma, we see that for any fixed $t\ge 0$.
\begin{eqnarray*}
\liminf_{n\to\infty}\,\frac{H\left(P_t(\eta,\cdot)\right)}{n} & \ge & \mathbb{H}\left(X(t)|X(0)\right),
\end{eqnarray*}
where $\mathbb{H}\left(Y|X\right)$ denotes the conditional entropy of $Y$ given $X$ defined by
\begin{eqnarray}
\mathbb{H}\left(Y|X\right) & := &
 \sum_{k,\ell}\PP\left(X=k,Y=\ell\right)\log\frac{1}{\PP\left(Y=\ell|X=k\right)}.
\end{eqnarray}
On the other hand, since $|\Omega|=\cN(n,m)$ with $\cN(n,m)$ defined at (\ref{binomial}), the uniform law $\pi$ satisfies
\begin{eqnarray*}
\frac{H(\pi)}{n} & \xrightarrow[n\to\infty]{} & (1+\rho)\log(1+\rho) - \rho\log {\rho}.
\end{eqnarray*}
The right-hand side is  $H\left(\cG\left(\rho\right)\right)$. Since $D_{\textsc{kl}}\left(\mu||\pi\right)  = H(\pi)-H(\mu)$ for $\pi$  uniform, we conclude that
\begin{eqnarray*}
\limsup_{n\to\infty}\frac{D_{\textsc{kl}}\left(P_t(\eta,\cdot)||\pi\right)}{n} & \le & H\left(\cG(\rho)\right)-\mathbb{H}\left(X(t)|X(0)\right)
\end{eqnarray*}
for any fixed $t\ge 0$. The conclusion (\ref{conclude:entropy}) follows, since the left-hand side is a decreasing function of $t$ (this is a general fact, see e.g. \cite{MR2341319}) and the right-hand side can be made arbitrarily small by choosing $t$ large enough, thanks to Lemma \ref{lm:mixing} and Proposition \ref{pr:fluid} (recall that $\lambda=\rho$ here).

\section{Dissolution of the solid phase}
\label{sec:solid}
In this final section, we start by verifying that there is a unique (explicit) solution to the Cauchy problem (\ref{fluid:dense}), and then show that the latter describes the evolution of the solid phase in the sense of Proposition \ref{pr:densefluid}. We finally put things together to prove Theorem \ref{th:main}.
\subsection{Resolution of the main differential equation}
With the setting of Theorem \ref{th:main} in mind, we fix a sequence of numbers $u_1\ge u_2\ldots\ge 0$ such that
\begin{eqnarray}
\label{assume:summable}
\sum_{i=1}^\infty u_i & \le &  \rho,
\end{eqnarray}
and we consider the Cauchy problem (\ref{fluid:dense}), repeated here for convenience:
\begin{eqnarray*}
\bu_i(t) & = & \left(u_i -\int_0^t\frac{1}{1+\rho-\sum_{j=1}^\infty \bu_j(s)}ds\right)_+.
\end{eqnarray*}  By a \emph{solution} to this problem, we will here mean a collection $(\bu_i)_{i\ge 1}$,  where for each $i\ge 1$,  $\bu_i\colon \dR_+\to [0,u_i]$ is a measurable function such that the equation (\ref{fluid:dense}) holds for all $t\in\dR_+$. We first deal with the uniqueness, and will construct an explicit solution afterwards.
\begin{lemma}[Uniqueness]\label{lm:uniqueness}There is at most one one solution to (\ref{fluid:dense}).
\end{lemma}
\begin{proof}
For $t\ge 0$, we introduce the key quantity
\begin{eqnarray*}
r(t) & := & \sum_{i=1}^\infty{\bf 1}_{(u_i>\frac{t}{1+\rho})}.
\end{eqnarray*}
Note that the non-increasing function $t\mapsto r(t)$ may diverge at zero, but that by condition (\ref{assume:summable}),
\begin{eqnarray}
\label{integrability}
\int_0^\infty r(s)\,{\rm d}s & = & (1+\rho)\sum_{i=1}^\infty u_i \ \le \ \rho(1+\rho).
\end{eqnarray}
Now, let $(\bu_i)_{i\ge 1}$ and $(\bw_i)_{i\ge 1}$ be two solutions to (\ref{fluid:dense}), and define for all $t\ge 0$,
\begin{eqnarray*}
\Delta(t) & := & \sum_{i=1}^\infty\left|\bv_i(t)-\bw_i(t)\right|\ \ \in [0,\rho].
\end{eqnarray*}
From the equation (\ref{fluid:dense}), it readily follows that 
\begin{eqnarray*}
\left|\bu_i(t)-\bw_i(t)\right| & \le & \int_0^t\Delta(s)\,{\rm d}s.
\end{eqnarray*}
Moreover, the left-hand side is zero for $t\ge (1+\rho)u_i$, because (\ref{fluid:dense}) implies
\begin{eqnarray*}
\bu_i(t)\vee \bw_i(t) & \le & \left(u_i-\frac{t}{1+\rho}\right)_+.
\end{eqnarray*}
Summing over all $i\ge 1$, we deduce from these two observations that
\begin{eqnarray*}
\Delta(t) & \le & r(t)\int_0^t\Delta(s)\,{\rm d}s\ 
 \le \ \int_0^tr(s)\Delta(s)\,{\rm d}s,
\end{eqnarray*}
where the second line follows from the fact that $r$ is non-increasing. Thanks to the integrability of $r$ (\ref{integrability}) and the uniform bound on $\Delta$, Grönwall's Lemma now implies that $\Delta(t)=0$ for all $t\in\dR_+$.
\end{proof}
Let us now construct an explicit solution to (\ref{fluid:dense}). We start by setting, for each $i\ge 1$, 
\begin{eqnarray}
\label{times}
t_i & := & u_{i}\left(1+\rho+\frac{(i-1)u_{i}}2-\sum_{j=1}^{i-1}u_j\right)-\frac 12\sum_{j=i}^\infty u_j^2.
\end{eqnarray}
Our assumptions on $(u_j)_{j\ge 1}$ easily imply that $t_i \to 0$ as $i\to\infty$.
Note also that the sequence $(t_i)_{i\ge 1}$ is non-increasing, since for all $i\ge 1$,
\begin{eqnarray*}
t_i-t_{i+1} & = & (u_{i}-u_{i+1})\left(1+\rho-\sum_{j=1}^iu_j\right)+\frac{i}2\left(u_i^2-u_{i+1}^2\right) \ \ge \ 0. 
\end{eqnarray*}
We now define a function $f\colon \dR_+\to\dR_+$ as follows: 
\begin{eqnarray*}
f(t) & := &  
\left\{
\begin{array}{ll}
0 & \textrm{ if }t=0.\\
\sqrt{\frac{2(t-t_{i+1})}{i}+\left(\frac{1+\rho-\sum_{j=1}^iu_j}{i}+\gamma_{i+1}\right)^2}-\frac{1+\rho-\sum_{j=1}^iu_j}{i}
& \textrm{ if }t\in (t_{i+1},t_i]\\
\frac{t-t_1}{1+\rho} + u_1 & \textrm{ if } t\ge t_1.
\end{array}
\right.
\end{eqnarray*}
Let us use the convenient convention $t_0:=+\infty$. Then for each $i\ge 0$, $f$ is increasing and ${\cal C}^\infty$ on $(t_{i+1},t_i)$. Moreover, for $i\ge 1$ we have $f(t_i +)=u_i=f(t_i)$, so $f$ is in fact continuous and increasing on $(0,\infty)$. In particular, $\lim f(0+)$ exists and must be equal to $\lim_{i\to\infty} f(t_i)$, which is $0$ because $u_i\to 0$. Now, for each $i\ge 0$ and each $t\in(t_{i+1},t_i)$, we easily compute
\begin{eqnarray*}
f'(t) & = & \frac{1}{1+\rho-\sum_{j=1}^i(u_j-f(t))}\ = \ \frac{1}{1+\rho-\sum_{j=1}^\infty(u_j-f(t))_+},
\end{eqnarray*}
where the second equality follows from the fact that $f(t)<u_i$ if and only if $t<t_i$, by strict monotony. Consequently, we may safely write, for all $t\ge 0$,
\begin{eqnarray*}
f(t) & = & \int_0^t\frac{1}{1+\rho-\sum_{i=1}^\infty(u_i-f(s))_+}\,ds.
\end{eqnarray*}
Setting $\bu_i(t) = \left(u_i-f(t)\right)_+$ yields a well-defined solution to (\ref{fluid:dense}) -- the only one, by Lemma \ref{lm:uniqueness}.
\subsection{Relating the dissolution rate to the density of the solid phase}
In our mean-field setting, the dissolution rate of the solid phase is  the proportion $Q_0(t)$ of empty sites in the system. By metastability, the latter should  only depend on the total density of particles in the solid phase. The purpose of this section is to make this intuition rigorous. For simplicity, we will here assume that the solid phase is restricted to the region $\{1,\ldots,L\}$ for some fixed $L\ge 0$, i.e.
\begin{eqnarray}
\label{assume:finitary}
\max_{i\ge L+1}\eta_i & = & o(n).
\end{eqnarray}
Note that this property is then preserved by the dynamics: by Chernov's bound, the uniform downward drift of Proposition \ref{pr:uniform} ensures that for any $t,\varepsilon>0$,
\begin{eqnarray}
\label{key:uniform}
\PP\left(\exists i>L\colon \eta_i(t)\ge \varepsilon n\right) & \le & 2ne^{-n\theta\varepsilon}\left(1+\max_{i\ge L+1}e^{\theta\eta_{i}-\delta t}\right),
\end{eqnarray}
where we recall that the constant $\theta>0$ does not depend on $n$. Since the right-hand side is summable in $n$, we see that the solid phase remains restricted to the region $\{1,\ldots,L\}$: for any time $t=t(n)$,
\begin{eqnarray}
\label{cons:finitary}
\max_{i\ge L+1}\eta_i(t) & = & o(n),
\end{eqnarray}
almost-surely. In particular, the proportion of particles in the solid phase at time $t$ is $\sum_{i=1}^L\frac{\eta_i(t)}{n}+o(1)$. The main result of this section is the following relation between this number and $Q_0(t)$.
\begin{proposition}[Dissolution rate]\label{pr:finitary}In the regime (\ref{assume:sparse})-(\ref{assume:finitary}), we have for any fixed $s>0$,
\begin{eqnarray*}
\left|Q_0(ns)-\frac 1{1+\rho-\sum_{i=1}^L\frac{\eta_i(ns)}{n}}\right| & \xrightarrow[n\to\infty]{\PP} & 0.
\end{eqnarray*}
\end{proposition}
To prove this proposition, we will ``erase'' the solid phase so that the fast mixing result of Section \ref{sec:fast} becomes applicable, and then compare this truncated process to the original one. 
\begin{lemma}[Truncation]\label{lm:truncation}Fix $\eta\in\Omega$, and let $\widehat{\eta}$ be   obtained by \emph{emptying} the first $L$ sites, i.e. 
\begin{eqnarray*}
\widehat{\eta}_i & = & 
\left\{
\begin{array}{ll}
{\eta}_i & \textrm{ if } i>L\\
0 & \textrm{ if } i\le L.
\end{array}
\right.
\end{eqnarray*} 
Then, the zero-range processes   starting from $\eta$ and $\widehat{\eta}$ can be coupled in such a way that their respective empirical profiles $(Q(t)\colon t\ge 0)$ and $(\widehat{Q}(t)\colon t\ge 0)$ satisfy, for any horizon $T\ge 0$,
\begin{eqnarray*}
\EE\left[\sup_{t\in[0,T]}\left\|\widehat{Q}(t) - Q(t)\right\| \right] & \le & \frac{2L(1+T)}n.
\end{eqnarray*}
\end{lemma}
\begin{proof}
Recall the standard construction of the zero-range process using an independent, rate$-\frac 1n$ Poisson point process $\Xi_{i\to j}$ for each source-destination pair $(i,j)\in[n]\times [n]$: the successive points indicate the times at which the source attempts to send a particle out to the destination, and the jump is allowed if and only if the source is not empty. We may couple the processes starting from $\eta$ and $\widehat{\eta}$ by simply using the same underlying Poisson clocks for both processes. We then have
\begin{eqnarray*}
\widehat{\eta}(t)&  \le & \eta(t),
\end{eqnarray*}
for all $t\ge 0$. Indeed, this inequality is true at time $t=0$ by construction, and it is preserved by the dynamics because any jump that is allowed for the left-hand side must also be allowed for the right-hand side. 
In particular, this implies that
\begin{eqnarray*}
\sum_{i=L+1}^n\left|\eta_i(t)-\widehat{\eta}_i(t)\right| & = & \sum_{i=L+1}^n\eta_i(t)-\sum_{i=L+1}^n\widehat{\eta}_i(t).
\end{eqnarray*}
The right-hand side equals zero at time $t=0$, and then the only clock rings that may increment it (by $1$ unit each time) are those whose source $i$ is in $\{1,\ldots,L\}$. Over the time interval $[0,T]$, the total number of such rings is just a Poisson random variable with mean $LT$, and hence
 \begin{eqnarray}
 \label{lips}
\EE\left[\sup_{t\in[0,T]}\sum_{i=L+1}^n\left|\widehat{\eta}_i(t)-\eta_i(t)\right|\right] & \le & LT.
\end{eqnarray}
Now, observe that for each $k\in\dZ_+$, we have
\begin{eqnarray*}
\left|Q_k(t)-\widehat{Q}_k(t)\right| & = & \left|\frac{1}{n}\sum_{i=1}^n{\bf 1}_{(\eta_i(t)=k)}-\frac{1}{n}\sum_{i=1}^n{\bf 1}_{(\widehat{\eta}_i(t)=k)}
\right|\  \le \  \frac{2}{n}\sum_{i=1}^n{\bf 1}_{\left(\eta_i(t)\ne\widehat{\eta}_i(t)\right)}{\bf 1}_{(\eta_i=k)}.
\end{eqnarray*}
Summing over $k$, we deduce that
\begin{eqnarray*}
\left\|Q(t)-\widehat{Q}(t)\right\| & \le & \frac{2}{n}\sum_{i=1}^n{\bf 1}_{\left(\eta_i(t)\ne\widehat{\eta}_i(t)\right)}
\ \le \ \frac{2}{n}\left(L+\sum_{i=L+1}^n\left|\eta_i(t)-\widehat{\eta}_i(t)\right|\right),
\end{eqnarray*}
and the claim now readily follows from (\ref{lips}).
\end{proof}
\begin{remark}[Stochastic regularity]\label{remark:smooth} The above construction of the zero-range process enjoys another useful property: letting $\Xi_i=\sum_j \Xi_{j\to i}+\Xi_{i\to j}$ denote the clock process for arrivals and departures on site $i$ -- which is a Poisson point process of intensity $2$ --, we clearly have for all $0\le s\le t$, 
\begin{eqnarray}
\left|\eta_i(t)-\eta_i(s)\right| & \le & \Xi_{i}\left(\left[s,t\right]\right).
\end{eqnarray}
\end{remark}

\begin{proof}[Proof of Proposition \ref{pr:finitary}]If $\eta$ satisfies (\ref{assume:finitary}), then its truncation $\widehat{\eta}$ is completely liquid in the sense of (\ref{assume:liquid}). By Proposition \ref{pr:fast}, this ensures that the zero-range process $(\widehat{\eta}(t)\colon t\ge 0)$ mixes in time $o(n)$. In particular, for fixed $\varepsilon>0$, the empirical profile $\widehat{Q}(n\varepsilon)$ must satisfy (\ref{equilibrium}) but with $m-(\eta_1+\cdots+\eta_L)$ particles instead of $m$, i.e.
\begin{eqnarray*}
\EE\left[\left\|\widehat{Q}(n\varepsilon)-\cG\left(\frac{m}{n}-\sum_{i=1}^L\frac{{\eta}_i}{n}\right)\right\|\right] & \xrightarrow[n\to\infty]{} & 0.
\end{eqnarray*}
On the other hand, under the coupling of Lemma \ref{lm:truncation}, we have 
\begin{eqnarray*}
\EE\left[\left\| {Q}(n\varepsilon)-\widehat{Q}(n\varepsilon)\right\|\right] & \le & 2L\varepsilon+\frac{2L}{n}.
\end{eqnarray*}
Finally, Remark \ref{remark:smooth} implies that
\begin{eqnarray*}
\EE\left[\left|\sum_{i=1}^L\frac{{\eta}_i}{n}-\sum_{i=1}^L\frac{{\eta}_i(n\varepsilon)}{n}\right|\right] & \le & 2L\varepsilon.
\end{eqnarray*}
Combining these three estimates, we easily deduce that
\begin{eqnarray*}
\limsup_{n\to\infty} \EE\left[\left|Q_0(n\varepsilon)-\frac 1{1+\rho-\sum_{i=1}^L\frac{\eta_i(n\varepsilon)}n}\right|\right] & \le & 4L\varepsilon.
\end{eqnarray*}
This seems rather weak compared to what we want to establish. However by the Markov property and (\ref{cons:finitary}), the result also applies to the shifted time $n\varepsilon+t(n)$, for any choice of $t(n)\ge 0$. Choosing $t(n)=(s-\varepsilon)n$ for fixed $s>0$ yields
 \begin{eqnarray*}
\limsup_{n\to\infty} \EE\left[\left|Q_0(ns)-\frac 1{1+\rho-\sum_{i=1}^L\frac{\eta_i(ns)}n}\right|\right] & \le &4L\varepsilon.
\end{eqnarray*}
Since this is valid for any choice $0<\varepsilon\le s$, the result follows. 
\end{proof}
\subsection{Tightness and convergence}
We are now in position to prove Proposition \ref{pr:densefluid}. We first establish a weak form of it, namely that in the finitely supported case where
\begin{eqnarray}
L & := & \sup\{i\ge 1\colon u_i >0\}
\end{eqnarray}
is finite, we have $U_i^n(t)\to \bu_i(t)$ in probability for fixed $i,t\ge 0$, with the convenient  short-hand 
\begin{eqnarray}
U_i^n(t) & := & \frac{\eta_i(nt)}{n}.
\end{eqnarray}
\begin{proof}[Proof of the weak form]We proceed by induction over $L$. The base case $L=0$ is trivial, by  (\ref{cons:finitary}). To move from $L-1$ to $L$, we only need to establish the convergence 
\begin{eqnarray*}
\left|U_i^n(t)-\bu_i(t)\right| & \xrightarrow[n\to\infty]{\PP} &  0
\end{eqnarray*}
for $i\le L$ and $t\in [0,T]$, where $T=t_L$ is the time at which the non-increasing function $\bu_L$ reaches $0$, as defined in (\ref{times}). In this range, we have
\begin{eqnarray}
\bu_i(t) & = & u_i-\int_0^t\frac{1}{1+\rho-\sum_{j=1}^L\bu_i(s)}ds.
\end{eqnarray}
On the other hand, by Dynkin's formula, we have the decomposition
\begin{eqnarray}
U_i^n(t) & = & U_i^n(0) -\int_0^{t}Q_0(ns)ds + \int_0^t{\bf 1}_{(U_i^n(s)=0)}ds+M_i^n(t),
\end{eqnarray}
where $M_i^n$ is a martingale. We will show that the right-hand sides of these two equations are \emph{close to each other} as $n\to\infty$, by a term-by-term comparison. First, we have
\begin{eqnarray}
U_i^n(0) & \xrightarrow[n\to\infty]{} & u_i,
\end{eqnarray}
by assumption (\ref{assume:profile}).
Second, Proposition \ref{pr:finitary} readily implies that
\begin{eqnarray*}
\sup_{t\in[0,T]}\left|\int_0^{t}Q_0(ns)\,{\rm d}s - \int_0^{t}\frac{1}{1+\rho-\sum_{i=1}^LU_i^n(s)}\,{\rm d}s\right|  & \xrightarrow[n\to\infty]{\PP} & 0.
\end{eqnarray*}
Third, setting $\Delta_n(t):=\max_{i\le L}|U_i^n(t)-u_i(t)|$, we have
\begin{eqnarray*}
\sup_{t\in[0,T]}\left|\int_0^{t}\frac{1}{1+\rho-\sum_{i=1}^Lu_i(s)}-\int_0^{t}\frac{1}{1+\rho-\sum_{i=1}^LU_i^n(s)}ds\right|  & \le & L\int_0^T \Delta_n(t)\,{\rm d}t.
\end{eqnarray*}
Fourth, observe that if $U_i^n(s)=0$ for some $s\le t$ then $\Delta_n(s)\ge u_i(s)\ge u_L(t)$, so that
\begin{eqnarray}
\label{localtime}
\int_0^t{\bf 1}_{(U_i^n(s)=0)}\,{\rm d}s & \le & \frac 1{u_L(t)}\int_0^t{\Delta_n(s)}ds.
\end{eqnarray}
Finally, since $U_i^n(t)$ makes jumps of  size $\frac{1}{n}$ at rate at most $2n$, we have 
$\EE\left[\left(M_i^n(T)\right)^2\right] \le  \frac{2T}{n}$ which, by Doob's maximal inequality, implies that 
\begin{eqnarray*}
\sup_{t\in[0,T]}|M_i^n(t)| & \xrightarrow[n\to\infty]{\PP} & 0.
\end{eqnarray*}
By Gr\"onwall's Lemma, we conclude that
\begin{eqnarray}
\label{gron}
\Delta_n(t) & \le & o_\PP(1)\exp\left\{Lt+\frac{t}{u_L(t)}\right\},
\end{eqnarray}
where the term $o_\PP(1)$ does not depend on $t$ and tends to $0$ in probability as $n\to\infty$. For fixed $t<T$, this already implies that $\Delta_n(t)\to 0$ in probability, because $u_L(t)>0$. In particular,  we may now go back to (\ref{localtime}) and improve it to 
\begin{eqnarray*}
\int_0^T{\bf 1}_{(U_i^n(s)=0)}\,{\rm d}s & \xrightarrow[n\to\infty]{L^1} & 0.
\end{eqnarray*}
This improvement suppresses the term $\frac{t}{u_L(t)}$ in (\ref{gron}), and the conclusion follows.
\end{proof}

\begin{proof}[Proof of Proposition \ref{pr:densefluid}]
Remark \ref{remark:smooth} is more than enough to ensure, for each $i\ge 1$,  the tightness (as $n$ varies) of $\left(U_i^n(t)\colon t\ge 0\right)$ in the Skorokhod space $D(\dR_+,\dR)$, and the almost-sure continuity of any sub-sequential limit. By diagonal extraction, we may find a subsequence along which
\begin{eqnarray*}
\left(U_1^{n},U_2^n,\ldots\right) & \xrightarrow[]{d} & \left(U_1^\star,U_2^\star,\ldots\right),
\end{eqnarray*}
weakly with respect to the product topology, with each coordinate being equipped with the topology of uniform convergence on compact sets. We already know that each limiting coordinate $U_i^\star$ belongs to $\cal C(\dR_+,\dR_+)$ with probability $1$, and our task boils down to proving that necessarily, $$U_i^\star(t)=\bu_i(t)$$ almost-surely, for each $i\ge 1$ and each $t\in\dR_+$.
Note that this is true at time $t=0$, by our assumption. Now let $s>0$. Since $u_i\le \frac{\rho}{i}$, it follows from Proposition \ref{pr:uniform} that $U^\star_i(s)=0$ for all $i>L:= \frac{\rho\theta}{\delta s}$. We may thus apply the weak form at time $s$ to deduce that necessarily,
\begin{eqnarray*}
U_i^\star(t) & = & \left(U_i^\star(s) -\int_s^t\frac{1}{1+\rho-\sum_{k=1}^\infty U_k^\star(u)}\,{\rm d}u\right)_+,
\end{eqnarray*}  
 for all $t\ge s$ and all $i\ge 1$. Letting $s\to 0$, we see that $(U_i^\star)_{i\ge 1}$ satisfies  (\ref{fluid:dense}) with probability $1$, and the uniqueness in Lemma \ref{lm:uniqueness} concludes the proof.
\end{proof}
\subsection{Putting things together: proof of the main result}
As a corollary of the above analysis, we obtain the following explicit value for the dissolution time. 
\begin{corollary}[Maximum occupancy]
In the regime (\ref{assume:sparse})-(\ref{assume:order})-(\ref{assume:profile}), we have for any fixed $t\ge 0$  
\begin{eqnarray}
\max_{1\le i\le n}\frac{\eta_i(nt)}{n} & \xrightarrow[n\to\infty]{\PP} & \bu_1(t).
\end{eqnarray}
Moreover, the function $\bu_1$ satisfies
\begin{eqnarray}
\bu_1(t)>0 & \Longleftrightarrow & t<t_1:=(1+\rho)u_1-\frac 12\sum_{i=1}^\infty u_i^2.
\end{eqnarray}
\end{corollary}
\begin{proof}
If the maximum were taken over $1\le i\le L$ for some fixed $L$, then the result would be a direct consequence of Proposition \ref{pr:densefluid}. The only point that needs to be justified is the fact that $\max_{1\le i\le n}\frac{\eta_i(nt)}{n}=\max_{1\le i\le L}\frac{\eta_i(nt)}{n}$ provided $L$ is chosen large enough. This fact is a clear consequence of (\ref{key:uniform}), together with the observation that $\eta_{L+1}\le \frac{\rho_\star}{L+1}$.
\end{proof}
This is all we need to complete the proof of Theorem \ref{th:main}.  We split the argument into two parts.  
\begin{proof}[Proof of the upper-bound]
If $t>t_1$ then $\bu_1(t)=0$, so Proposition \ref{pr:fast} ensures fast mixing from the random configuration  $\xi:=\eta(nt)$, i.e. 
\begin{eqnarray*}
\frac{\tmix(\xi;\varepsilon)}{n} & \xrightarrow[n\to\infty]{\PP} & 0.
\end{eqnarray*}
Explicitating the definitions, this means that for any fixed $s>0$,
\begin{eqnarray*}
\sup_{A\subseteq\Omega}\left|P_{ns}(\xi,A)-\pi(A)\right| & \xrightarrow[n\to\infty]{\PP} & 0.
\end{eqnarray*}
Taking expectations and noting that $\EE[P_{ns}(\xi,A)]=P_{n(t+s)}(\eta,A)$ by  Markov's property, we obtain
\begin{eqnarray*}
\sup_{A\subseteq\Omega}\left|P_{n(t+s)}(\eta,A)-\pi(A)\right| & \xrightarrow[n\to\infty]{} & 0.
\end{eqnarray*}
In other words, for any fixed $\varepsilon\in(0,1)$,
\begin{eqnarray*}
\limsup_{n\to\infty}\left\{\frac{\tmix(\eta;\varepsilon)}{n}\right\} & \le & t+s.
\end{eqnarray*}
Since $s$ and $t$ can be chosen arbitrarily close to $0$ and $t_1$ respectively, the upper-bound is proved. 
\end{proof}
\begin{proof}[Proof of the lower-bound]Our distinguishing event will be 
\begin{eqnarray}
 A & := & \left\{\eta\in\Omega\colon \max_{1\le i\le n}\eta_i \ge \sqrt{n}\right\}.
 \end{eqnarray} 
We note for clarity that the choice $\sqrt{n}$ is irrelevant here: any $k=k(n)$ satisfying $\log n\ll k(n)\ll n$ will work. If $t<t_1$, then $\bu_1(t)>0$, and so the above corollary shows that
\begin{eqnarray*}
P_{nt}(\eta,A) & \xrightarrow[n\to\infty]{} & 1.
\end{eqnarray*}
On the other hand, we have seen in Section \ref{sec:liquid} that when $\xi$ has the uniform law $\pi$, 
\begin{eqnarray}
\PP\left(\xi_i=k\right) & = & \frac{\cN(n-1,m-1)}{\cN(n,m)}.
\end{eqnarray}
The right-hand side is at most $\left(\frac{m}{m+n-1}\right)^k=\left(\frac{\rho}{\rho+1}+o(1)\right)^k$, which is $o(n^{-2})$ when $k\ge\alpha\log n$ with $\alpha$ large enough. Summing over all possible choices for $i$ and $k$, we deduce that
\begin{eqnarray*}
\pi(A) & \xrightarrow[n\to\infty]{} & 0.
 \end{eqnarray*}
These two estimates prove that $\left\|P_{nt}(\eta,A)-\pi\right\|_{\textsc{tv}}\to 0$, or equivalently, that for all $\varepsilon\in(0,1)$, 
\begin{eqnarray*}
\liminf_{n\to\infty}\left\{\frac{\tmix(\eta;\varepsilon)}{n}\right\} & \ge & t.
\end{eqnarray*}
Since $t$ can be chosen arbitrarily close to $t_1$, the lower-bound follows.
\end{proof}
\bibliographystyle{plain}
\bibliography{draft}

\begin{thebibliography}{10}

\bibitem{MR770418}
David Aldous.
\newblock Random walks on finite groups and rapidly mixing {M}arkov chains.
\newblock In {\em Seminar on probability, {XVII}}, volume 986 of {\em Lecture
  Notes in Math.}, pages 243--297. Springer, Berlin, 1983.

\bibitem{MR841111}
David Aldous and Persi Diaconis.
\newblock Shuffling cards and stopping times.
\newblock {\em Amer. Math. Monthly}, 93(5):333--348, 1986.

\bibitem{MR2322692}
Pietro Caputo and Gustavo Posta.
\newblock Entropy dissipation estimates in a zero-range dynamics.
\newblock {\em Probab. Theory Related Fields}, 139(1-2):65--87, 2007.

\bibitem{MR2239987}
Thomas~M. Cover and Joy~A. Thomas.
\newblock {\em Elements of information theory}.
\newblock Wiley-Interscience [John Wiley \& Sons], Hoboken, NJ, second edition,
  2006.

\bibitem{MR1374011}
Persi Diaconis.
\newblock The cutoff phenomenon in finite {M}arkov chains.
\newblock {\em Proc. Nat. Acad. Sci. U.S.A.}, 93(4):1659--1664, 1996.

\bibitem{MR626813}
Persi Diaconis and Mehrdad Shahshahani.
\newblock Generating a random permutation with random transpositions.
\newblock {\em Z. Wahrsch. Verw. Gebiete}, 57(2):159--179, 1981.

\bibitem{MR838085}
Stewart~N. Ethier and Thomas~G. Kurtz.
\newblock {\em Markov processes}.
\newblock Wiley Series in Probability and Mathematical Statistics: Probability
  and Mathematical Statistics. John Wiley \& Sons, Inc., New York, 1986.
\newblock Characterization and convergence.

\bibitem{EVANS2000}
M.~R. Evans.
\newblock {Phase transitions in one-dimensional nonequilibrium systems}.
\newblock {\em {Brazilian Journal of Physics}}, 30:42 -- 57, 03 2000.

\bibitem{MR2145800}
M.~R. Evans and T.~Hanney.
\newblock Nonequilibrium statistical mechanics of the zero-range process and
  related models.
\newblock {\em J. Phys. A}, 38(19):R195--R240, 2005.

\bibitem{MR3320314}
S.~Ganguly, E.~Lubetzky, and F.~Martinelli.
\newblock Cutoff for the east process.
\newblock {\em Comm. Math. Phys.}, 335(3):1287--1322, 2015.

\bibitem{MR2521877}
Benjamin~T. Graham.
\newblock Rate of relaxation for a mean-field zero-range process.
\newblock {\em Ann. Appl. Probab.}, 19(2):497--520, 2009.

\bibitem{MR3551201}
Hubert Lacoin.
\newblock The cutoff profile for the simple exclusion process on the circle.
\newblock {\em Ann. Probab.}, 44(5):3399--3430, 2016.

\bibitem{MR3474475}
Hubert Lacoin.
\newblock Mixing time and cutoff for the adjacent transposition shuffle and the
  simple exclusion.
\newblock {\em Ann. Probab.}, 44(2):1426--1487, 2016.

\bibitem{MR3689972}
Hubert Lacoin.
\newblock The simple exclusion process on the circle has a diffusive cutoff
  window.
\newblock {\em Ann. Inst. Henri Poincar\'e Probab. Stat.}, 53(3):1402--1437,
  2017.

\bibitem{MR2869447}
Hubert Lacoin and R\'emi Leblond.
\newblock Cutoff phenomenon for the simple exclusion process on the complete
  graph.
\newblock {\em ALEA Lat. Am. J. Probab. Math. Stat.}, 8:285--301, 2011.

\bibitem{MR2550363}
David~A. Levin, Malwina~J. Luczak, and Yuval Peres.
\newblock Glauber dynamics for the mean-field {I}sing model: cut-off, critical
  power law, and metastability.
\newblock {\em Probab. Theory Related Fields}, 146(1-2):223--265, 2010.

\bibitem{MR3726904}
David~A. Levin, Yuval Peres, and Elizabeth~L. Wilmer.
\newblock {\em Markov chains and mixing times}.
\newblock American Mathematical Society, Providence, RI, 2017.
\newblock Second edition of [ MR2466937], With a chapter on ``Coupling from the
  past'' by James G. Propp and David B. Wilson.

\bibitem{MR2108619}
Thomas~M. Liggett.
\newblock {\em Interacting particle systems}.
\newblock Classics in Mathematics. Springer-Verlag, Berlin, 2005.
\newblock Reprint of the 1985 original.

\bibitem{MR3020173}
Eyal Lubetzky and Allan Sly.
\newblock Cutoff for the {I}sing model on the lattice.
\newblock {\em Invent. Math.}, 191(3):719--755, 2013.

\bibitem{MR3193965}
Eyal Lubetzky and Allan Sly.
\newblock Cutoff for general spin systems with arbitrary boundary conditions.
\newblock {\em Comm. Pure Appl. Math.}, 67(6):982--1027, 2014.

\bibitem{MR2341319}
Ravi Montenegro and Prasad Tetali.
\newblock Mathematical aspects of mixing times in {M}arkov chains.
\newblock {\em Found. Trends Theor. Comput. Sci.}, 1(3):x+121, 2006.

\bibitem{MR2271475}
Ben Morris.
\newblock Spectral gap for the zero range process with constant rate.
\newblock {\em Ann. Probab.}, 34(5):1645--1664, 2006.

\bibitem{MR0268959}
Frank Spitzer.
\newblock Interaction of {M}arkov processes.
\newblock {\em Advances in Math.}, 5:246--290 (1970), 1970.

\bibitem{MR1108185}
Alain-Sol Sznitman.
\newblock Topics in propagation of chaos.
\newblock In {\em \'Ecole d'\'Et\'e de {P}robabilit\'es de {S}aint-{F}lour
  {XIX}---1989}, volume 1464 of {\em Lecture Notes in Math.}, pages 165--251.
  Springer, Berlin, 1991.

\end{thebibliography}
\end{document}